\documentclass[11pt,a4paper]{article}
\usepackage[latin1]{inputenc}
\usepackage{indentfirst}
\usepackage{amsmath}
\usepackage{amsfonts}
\usepackage{amsthm}
\usepackage{amssymb}
\usepackage{graphics}
\usepackage{graphicx}
\usepackage{multicol}
\usepackage{color}
\usepackage{url}

\def\N{\mathbb{N}}

\def\P{\mathbb{P}}
\def\E{\mathbb{E}}

\def\R{\mathbb{R}}

\def\K{\mathcal{K}}

\def\d{\mathfrak{f}}%%\mathbf{f}}

\def\tauxmort{\beta}%{\beta}
\def\tauxmorti{\beta_i}
\def\tauxnaissance{\gamma}%{\gamma}
\def\tauxnaissancei{\gamma_i}
\def\tauxinter{\alpha}%{\alpha}%interaction
\def\tauxinteri{\alpha_i}
\def\pfff{u}%interaction
\def\tauxgrowth{\lambda }%{\lambda}%taux de croissance
\def\lyap{\mathbf{V}}%Fonction de Lyap
\newcommand{\be} {\begin{equation}}
\newcommand{\ee} {\end{equation}}
\newcommand{\bea} {\begin{eqnarray}}
\newcommand{\eea} {\end{eqnarray}}
\newcommand{\Bea} {\begin{eqnarray*}}
\newcommand{\Eea} {\end{eqnarray*}}

\newtheorem{prop}{Proposition}[section]

\newtheorem{assumption}[prop]{Assumption}
\newtheorem{lem}[prop]{Lemma}
\newtheorem{thm}[prop]{Theorem}

\def\cmb#1{\marginpar{\raggedright\tiny{\textcolor{red}{Bertrand} : \textcolor{blue}{#1}}}}
\begin{document}
\title{From the distributions of  times of interactions to \\
preys and predators dynamical systems}

% and space Models to discrete time and continuous space Models (via large populations) and then continous time and space Models
%}

\author{Vincent Bansaye\footnote{CMAP, Ecole polytechnique,  Palaiseau} \,  and  Bertand Cloez\footnote{INRAe, Mistea, Montpellier}}
\maketitle \vspace{3cm}

\begin{abstract} 
We consider a stochastic individual based model
where each predator searches  during a random time and then manipulates its prey or rests. The time distributions may be non-exponential.
An age structure allows to describe these interactions  and get a Markovian setting. The process is characterized by a measure-valued stochastic differential equation. We  prove averaging results in this infinite dimensional setting and get the convergence of  the slow-fast  macroscopic prey predator process to a two dimensional dynamical system. We recover classical functional responses. We also get new forms arising in particular when births and deaths of predators are affected by the lack of food.
\end{abstract}

\tableofcontents%

\section{Introduction}

Functional responses are widely  used to quantify interactions  between species  in  ecology.   The way  functional responses arise at the macroscopic level and describe population dynamics
or evolution is  a fundamental issue for species conservation or
statistical inference of parameters. Indeed, their  form  influences
 the stability properties of dynamics, their  long time behavior or speed of convergence.  The link between individual behavior and macroscopic dynamics has attracted lots of attention for chemical reactions and population dynamics from the works of Michael and Menten. \\
 
 Macroscopic derivation from individual based model rely in general
on  a large population approximation of finite dimensional  Markov processes  describing the number of individuals of each species, possibly structured in status (searching, handling...), space or size.
In this setting, Kurtz and Popovic \cite{KKP} obtain the classical Michaelis Menten and Holling functional responses in limiting dynamical systems and fluctuations of processes around these limits. In our context of prey-pradators interactions, let us  mention \cite{DS} which starts from a stochastic individual based model. They derive
 a finite dimensional Markov chain
and   convergence to   ODEs involving the classical functional responses.
In \cite{CKBG},  a simple decision  tree  based on game-theoretical  approach  response is developed.   Similarly, random walks and Poisson type process are used in \cite{AKF} to describe functional responses. The  reduced model counting only the total number 
of preys and the total number of predators, without distinguishing their status,  is also classically 
derived directly from the macroscopic ODEs \cite{JKT,BBRS, HBR}. Again, it uses a slow-fast scaling
  and the associated quasi-steady-state approximation.  These Markov settings allow for justification of macroscopic equations in a context of absence of memory of interactions. Indeed, the time for associated interactions are then exponentially distributed, potentially up to the addition of the relevant successive state to describe the interaction. 

Random times involved  in ecological or biological interactions are in general non-exponentially distributed, see \cite{Duijns, BBC18} and references therein.  Indeed, handling or manipulation times  may  have small standard deviations compared to the mean, while exponential distribution forces the value of variance once the mean is fixed. Besides, as far as we see, these  times  seem to be distributed with one mode. Finally, foraging suggests that the probability of finding a prey eventually increases with searching time for a given density of preys. The aim of the paper is to consider general distribution  for the times  describing  interactions. 
%We determine the microscopic hypotheses describing the classical results and relax these hypotheses by adding memory, while keeping a regeneration hypothesis; \textit{i.e.}  independence between successive times hold. 
We extend approximation results 
relying on absence of memory and obtain a reduced model.
%  a two dimensional dynamical system whose parameter  are given by the  
 We also obtain new features due to the fact that mortality depends on prey consumption and life length is not exponentially distributed.
Following in particular  \cite{BBC18} and references therein, we  model the interaction
by a renewal process for each predator, with two status. Each predator successively searches during a random time and then manipulates during an other random time, which may include rest or other interactions.
We assume that these time distributions  admit a  density with respect to Lebesgue measure and  density dependence. Extension of the current approach to more than two status for predators would be straightforward. \\

Let us  first describe  informally the model. We write $n_1\in \N$ the number of  preys and $n_2\in \N$ the number of predators. Predators 
then search preys during a  random time distributed as a random variable $T_S(n_1)$. Typically the more $n_1$ is large, the smaller $T_S(n_1)$ should be. At the end of this time, one prey is caught and the population of preys becomes $n_1-1$. The predator changes its status and now manipulates during  a time distributed as  $T_M(n_1-1)$. Several predators follow simultaneously and
 independently this dynamics, but they live with a common number of preys and impact each other through this common resource.
 Besides, each predator gives birth and dies with respective individual rates  $\tauxnaissance_r(u)$ and $\tauxmort_r(u)$, which depends on their status  $r\in \{S,M\}$ and the time $u$ from which they are in this status. Typically, the fact that the predator does not find a prey make  its death rate  $\tauxmort_S(u)$ increase with $u$. Preys also give birth and die,   with fixed rates $\tauxnaissance$ and $\tauxmort$. \\
 We assume that preys are at scale $K_1$ and predators at scale $K_2$ and  that $K_1\gg K_2$. That means that preys are much more numerous than predators. 
  A slow-fast dynamic is considered : the  time scale  of prey-predator interactions  is short compared to the time scale of birth and death  of predators and preys. It 
  means that each predator  eats many preys during its life and, if a prey is not eaten by a predator then its life length is comparable to the ones of predators.  After scaling, we show, that the couple  of stochastic  processes describing the quantities of preys of predators converge in law  in $\mathbb D([0,\infty), \R_+^2)$ as $K_1,K_2$ tend to infinity, to the unique solution  $(x,y)$ of  an ordinary differential equation:
\begin{equation*}
\left\{
\begin{array}{rcl}
x'(t)&=&(\tauxnaissance - \tauxmort)x(t)-y(t)\phi(x(t)),\\
y'(t)&=&y(t)\psi(x(t)),
\end{array}
\right.
\end{equation*}
where
\begin{align}
\label{eq:phi-intro}
\phi(x)&=  \frac{1}{\E[T_S(x)  + T_M(x)]},
\end{align}
and
\begin{align}
\label{eq:psi-intro}
\psi(x)&=\frac{\E\left[ \int_0^{T_S(x)} (\tauxnaissance_S(u) - \tauxmort_S(u)) du \, + \, \int_0^{T_M(x)} (\tauxnaissance_M(u) - \tauxmort_M(u)) du\right]}{\E[T_S(x)  + T_M(x)]}.
%}{\E[T_S(x)]+\E[T_M(x)]}.
\end{align}
This limit theorem will be illustrated both by classical and new functional responses  in Section~\ref{se:exemple}. We observe that the response $\phi$ of preys due to predatory  is only sensitive to mean time of interactions. It thus extends the exponential case to more general distribution. At this macroscopic  scale,  for the population of preys, the distribution of times involved in interactions  plays a role only through its mean.   In contrast, the growth rate $\psi$ of the population of predators is in general sensitive to the other characteristics
of the distribution. We add  that the distribution of time of interactions should also  impact the dynamics of the population of preys at a second order, \textit{i.e.} for fluctuations. This is relevant in particular when population are not too large and left for a future work.

The fact that the time of interactions is both density dependent and non-exponentially distributed leads us to extend the state space. This  procedure
to get the Markov setting is classical and consists here in an additional 
age structure.  We then
 exploit the generator and martingale problem and get also the age distribution of preys. The problem arising is then an averaging in infinite dimension. The strategy of proof follows the techniques developed in \cite{KKP} in finite dimension using the occupation measure. In infinite dimension, much less work has been done up to our knowledge. Let us mention \cite{MT12} which considers averaging with an age structure and has also inspired this work. Two main differences appear in our context : the age structure is due to interactions and the rates involved are not bounded, since tail distribution of times may for instance decrease faster than exponentially.   \\
We consider a punctual measure whose atoms give the status and the age  of predators, which is here the length of time since they have been in this status. Other relevant ages could be added, in particular the time from the birth. However, it seems superfluous at our stage. In our slow-fast dynamics, there is an averaging   phenomenon and the numbers of predators in each status are instantaneously at equilibrium. This enables to reduce the infinite-dimensional model to a two-dimensional system of equations. The averaging phenomenon in finite dimension is classical \cite{KKP,BKPR, C16, MT12}.
The reduction of the infinite setting to a finite one describing the number of preys and predators may be less. Following \cite{K92,KKP}, the occupation measure $\Gamma^K$, given by $\Gamma^K([0,t])=\int_0^t \delta_{Y^K_s} ds$  deals with the fast time component $Y^K$, here the predations. In our setting, $Y^K$ is the distribution of ages and status and is thus defined as a punctual measure. Instead of considering a measure whose atoms are punctual measures, we consider the mean measure $\Gamma^K([0,t])=\int_0^t Y^K_s ds$, which is enough for our purpose. Consequently, 
%in contrast with \cite{K92,KKP, BKPR,C16}, 
our measure $\Gamma^K$ will not degenerate to some measure of the form $\int_0^t \delta_{f(X_s)} ds$, for some function $f$, but tends to some specific distribution. \\
%In particular, the limiting number of preys alone $X^K$, a slow variable,  is not a Markov process but its dynamics only further depends on the limiting number of predators. \\
$\newline$

The paper is structured as follows. In Section~\ref{se:model}, we  define
the integer valued model without any time or  population size scaling. We characterize the process as the unique strong solution of a stochastic differential equation 
and give  first properties about its semimartingale decomposition. In particular, the key technical point is the control of the age distribution in our setting, by exploiting the local time associated to the renewal procedure. In Section~\ref{se:scaling}, we introduce the  scaled process
and state  our main result, Theorem~\ref{thm:main}. The result is proved by tightness and identification of the limit using the occupation measure. Finally, we end our work through several examples in Section~\ref{se:exemple}. \\
$\newline$
%\section{Definition et caract\'erisation du mod\`ele (sans \'echelle, juste interaction)}
%\emph{Quasi definition du processus pour mon ti Bertrand (pour une vraie definition EDS, le faire a la chi : faire vivre les masses de Dirac pour toute age des leur naissance et les retirer pour les temps ou elles sont mortes}
{\bf Notation}.  We write $a_{\infty}\in (0,+\infty]$ the maximal age and
$$
\mathcal X=\{S,M\}\times [0,a_\infty), \quad 
$$
the state space of predators endowed with the product $\sigma$-algebra.\\
We denote by  $\mathfrak{M}(\mathcal S)$ the set of finite  measures on any topological space $\mathcal S$ endowed with its Borel algebra. We endow $\mathfrak{M}(\mathcal S)$  with the narrow (or weak) topology: that is $\mu_n$ tends to $\mu$ if and only if for every continuous and bounded function $f$ on $\mathcal S$,
$$
\lim_{n\to \infty} \int_{\mathcal S} f d\mu_n = \int_{\mathcal S} f d\mu.
$$
For $r\in\{S,M\}$, we write $\overline{r}$  the complementary status of $r$, \textit{i.e.} the unique element of $\{S,M\}\setminus\{r\}$.\\
We finally denote by
$\mathcal C^{1,b}( \mathcal{X})$ (resp. $C^{1,b}(\mathcal U\times \mathcal X)$ and 
$ \mathcal C^{1,b}([0,a_{\infty}))$)
 the space of measurable and bounded functions 
from $\{S,M\}\times [0,a_\infty)$ (resp.  $\mathcal U \times \{S,M\}\times [0,a_\infty)$ and $[0,a_\infty))$ to $\R$ such that 
 $f$ is continuously differentiable with respect to its second (resp. third, resp. first) variable, with bounded derivative.

\section{The stochastic individual based model}
\label{se:model}

In this section, we define the discrete individual based model using stochastic differential equations with jumps. The accelerated normalized process and its approximation will be studied in the next section using properties and estimates obtained here.\\
 Each predator is characterized by a status $r\in \{S,M\}$ and an age $a \in [0,a_{\infty})$ . To formalize conveniently this modeling, we  label each predator using classical Ulam-Harris-Neveu notation and describe the associated genealogical tree. The set of individuals is
$$\mathcal U =\mathbb \N\times \cup_{k\geq 0} \{1,2\}^k.$$
For short, we write  $u=u_0u_1\ldots u_k\in \mathcal U$ and $u$ then corresponds to an individual living
in generation $\vert u \vert =k$ and whose ancestor in generation $i$
is $u_0\ldots u_i$ for $0\leq i\leq k$. At each reproduction event, we assume for simplicity that every predator $u$ only gives birth to one predator and we label
the mother by $u1$ and its child by $u2$.
The population of predators alive at time $t$ is denoted by $\mathcal P(t)$ and is a subset of $\mathcal U$. For each predator $i\in \mathcal P(t)$, we write $r_i(t)\in \{S,M\}$ its status at time $t$, which indicate respectively that  it searches or manipulates.
We write $a_i(t)$ its age, namely the time from which it searches or manipulates.
Finally, we write $X(t)\in \N$ the number of preys at time $t$. \\
The state of the population  is then given by the process $$Z=(Z(t))_{t\geq 0}=(X(t), Y(t))_{t\geq 0},$$
 where the measure $Y$ describes the predators  and is given for all $t\geq 0$ by
$$
 Y(t)= \sum_{i\in \mathcal P(t)} \delta_{(i,r_i(t),a_i(t))}.
$$
 For any $t\geq 0$,  $Y(t)\in \mathfrak{M}(\mathcal{U} \times \mathcal X)$, where we recall
 that  $
\mathcal X=\{S,M\}\times [0,a_\infty).
$
Besides,  for any $U\subset \mathcal U$, the projected measure
$$Y(t,U,\{r\}, \cdot)=  \sum_{i\in \mathcal P(t), \, r_i(t)=r} \delta_{a_i(t)}$$ 
gives  the  collection of ages  of predators in status $r$ at time $t$, whose labels belong to $U$.
 The total number of predators at time $t$ is then $Y(t,\mathcal U,\{S,M\}, \mathbb{R}_+)$.\\

Let us now describe the population dynamic. For status $r \in\{S,M\}$ and in the presence of $x\in \mathbb N$ preys, we assume that the  time for interaction $T_{r}(x)$ is a random variable with support $[0,a_{\infty})$. We also assume that it admits a density, with respect to the Lebesgue measure,  $\d_{r}(x,\cdot)$. The associated jump rate $\tauxinter$ is defined  for $(a,x)\in \mathcal X$ by 
$$
\tauxinter_{r}(a,x)%=\tauxinter(\bullet, a,x)
=\frac{\d_{r}(a,x)}{\int_{[a,a_\infty)}\d_r(u,x)du},
$$
  It gives the rate at which a predator changes its status when its age in its current status is
  equal to $a$ and the number of preys is $x$.  If the predator was searching, it provokes a death of a prey.  We do not assume that these interactions rates $\alpha_r$ are lower and upper bounded. Indeed, for instance in the case when   the time of interaction  has a finite support ($a_{\infty}<\infty$) or a subexponential tail, it is not upperbounded, even for a given number of preys.    Let us 
  consider  some classical distribution that will be captured in our setting:
  % These ones will satisfy our assumptions :
\begin{itemize}
\item Exponential law: $\d(a,x)=\lambda(x) e^{-\lambda(x) a}$ and $\tauxinter(a,x)= \lambda(x)$, for some bounded function $\lambda$.
\item Log-normal distribution :  $\d(a,x)=\frac{1}{a\sigma(x)\sqrt{2\pi}}\exp\left(-\frac{(\log(a)-\mu(x))^2)}{2\sigma(x)^2}\right)$ and $\tauxinter$ has no explicit form.
\item Uniform law: $\d(a,x)= \mathbf{1}_{[0,1]}(a)$ and $\alpha(a,x) = (1-a)^{-1} \mathbf{1}_{[0,1)}$.
\item Pareto law : $\d(a,x)= k(x) (z(x)^{k(x)}/a)^{k(x)} \mathbf{1}_{a\geq z(x)}$ and $\alpha(a,x)=k(x)/a \mathbf{1}_{a \geq z(x)}$, for some bounded functions $k:\mathbb{N} \to (0,+\infty)$ and $z:\mathbb{N} \to (1,+\infty)$.
\end{itemize}

Finally, predators may give birth or die with respective measurable rates $a\mapsto\tauxnaissance_{r}(a)$ and $a\mapsto\tauxmort_{r}(a)$ which depend on their status $r\in \{S,M\}$ and their interaction age $a$. In particular, lack of nutrition affects survival and reproduction and $\tauxnaissance_{S}$ may be decreasing and $\tauxmort_{S}$ may be increasing.
For sake of simplicity and realism, we assume these rates are bounded.
 For the preys, birth and death rates are non-negative numbers denoted by $\tauxnaissance$ and $\tauxmort$.

\subsection{Existence and trajectorial representation}

Following for instance \cite{FM04, T06, BT10}, we construct and characterize 
$(Z(t))_{t\geq 0}$ as the unique strong solution of a stochastic differential equation. For every $i\in \mathcal{U}$, we let $\mathcal N^i$  be independent Poisson punctual point measures  on $\R_+^2$ with intensity the Lebesgue measure. These measures  provide the random times  when a predator changes its status between searching and manipulating.
We introduce also   independent Poisson punctual point measures $\mathcal M^i$ and $\mathcal Q$ on $\R_+^2$
with intensity  the Lebesgue measure. They are independent of $(\mathcal N^i, i\in \mathcal U)$ and describe births and deaths of preys and predators. For convenience, we write
$$ \tauxinteri(s)=\tauxinter_{r_i(s)}(a_i(s),X(s)), \quad \tauxnaissancei(s)=\tauxnaissance_{r_i(s)}( a_i(s)), \quad  \tauxmorti(s)= \tauxmort_{r_i(s)} (a_i(s)).$$
We consider the following equation for the evolution of the number of preys for
  $t\geq 0$, 
\begin{align}
X(t)&= X(0)-\int_0^t  \sum_{\substack{i\in \mathcal P(s-) \\ r_i(s-)=S}}  \int_{\mathbb{R}_+}\mathbf{1}_{\{u\leq \tauxinteri(s-)\}} \, \mathcal N^i(ds,  du)\nonumber \\
&\qquad \qquad +\int_0^t \int_{\mathbb{R}_+}  \left(\mathbf{1}_{\{u\leq \tauxnaissance X(s-)\}} - \mathbf{1}_{\{0 < u- \tauxnaissance X(s-)\leq \tauxmort X(s-)\}} \right)\, \mathcal Q(ds,  du). \label{S1}
\end{align}
Indeed, the number of preys  decreases when they are caught by a predator and also varies independently by  births and deaths. For every function $f \in \mathcal C^{1,b}(\mathcal U\times \mathcal X)$, we consider
\begin{align}
\langle Y(t),&f\rangle=\langle Y(0),f\rangle
+\int_0^t \sum_{i\in \mathcal P(s-)}\partial_a f (i,r_i(s-),a_i(s-)) \, ds
\nonumber \\
&\qquad + \int_0^t \sum_{i\in \mathcal P(s-)}\int_{\R_+} \mathbf{1}_{u\leq \tauxinteri(s-)}\, 
 Df(i,s-) \, \mathcal N^i(ds,  du) \nonumber\\
&\qquad +\int_0^t \sum_{i\in \mathcal P(s-)}\int_{\R_+} \Big(\mathbf{1}_{u\leq \tauxnaissancei(s-)}\,  \Delta f(i,r_i(s-),a_i(s-)) ) \nonumber\\
&\quad \qquad \qquad \qquad \qquad   - \mathbf{1}_{ 0< u- \tauxnaissancei(s-)\leq \tauxmorti(s-) } \, f(i,r_i(s-),a_i(s-))\Big) \mathcal Q^i(ds,  du), \label{S2}
\end{align}
where $\partial_a f $ stands for the partial derivative of $f$ with respect to the third variable and
$$ Df(i,s)= f(i,\overline{r_i}(s), 0)-f(i,r_i(s), a_i(s)); \quad \Delta f(i,r,a)= f(i1,r,a)+f(i2,M,0)-f(i,r,a).$$
Recall that $\overline{r}$ is the complementary status of $r$, \textit{i.e.} the unique element of $\{S,M\}\setminus\{r\}$.
In this equation, status of new born are supposed to be in the
manipulation state  $M$. This choice seems natural  but may sound somewhat arbitrary.
% Of course, we could have chosen a different modeling. 
However, more complex choices  would  require additional notations and should have no macroscopic impact. \\

Let us state the existence result and characterize the process using the previous stochastic differential equation. For convenience, we make the following boundedness and regularity assumptions, which are relevant for our purpose.

\begin{assumption}
\label{hyp:ci_ws}
% We fix an initial condition $Z(0)=(X(0),Y(0))\in \N\times \mathfrak M_{U\times \mathcal X}$ such that
There exists $a_0 \in (0,a_\infty)$ such that $ Y(0,\{S,M\},[a_0,a_\infty))=0$ a.s. \\
 Besides, for any $r\in \{S,M\}$ and $\K>0$,
$$\inf_{x\in [0, \K]}\E(T_r(x)) >0, \qquad \sup_{a\in[0,a_{\infty})} \left( \tauxnaissance_{r}(a)+\tauxmort_{r}(a) \right) <\infty $$
and $a\rightarrow \alpha_r(a,x)$ is continuous on $[0,a_{\infty})$ for any $x\in \N$.
\end{assumption}

\begin{prop}% Existence of a solution, and strong uniqueness
\label{prop:existe} Let $Z(0)=(X(0),Y(0))$ with $X(0)\in \N$ and $Y(0)$ being a punctual measure in $\mathfrak M(U\times \mathcal X)$ a.s. Under Assumption \ref{hyp:ci_ws},
the system of stochastic differential equations (\ref{S1}-\ref{S2}) admits a  unique strong solution $Z=(X,Y)$ in $\mathbb D([0,\infty), \N\times \mathfrak M(U\times \mathcal X))$
with initial condition $Z(0)$. 
\end{prop}

\begin{proof} 
The construction of the process and its uniqueness can be achieved iteratively, using the successive random times between each event, see for instance \cite{BM15}. 
The proof is classical and we just give its sketch. The only point to justify is that the successive times where an event (change of status, birth or death) occur do not accumulate. For that purpose, we proceed by a classical localization procedure and
introduce the hitting time $T_K=\inf\{ t : X(t)\geq K\}$.
Before time $T_K$, by Assumption~\ref{hyp:ci_ws}, mean time of change of status are lower bounded and  birth and death  rates of preys and predators are  upper bounded. So a.s.  no accumulation of jumps occurs. We need now to justify that $T_K$ tends a.s. to infinity as $K\rightarrow \infty$. It is achieved by dominating the process $X$ by a pure linear birth process (Yule  process) with birth rate \textit{per capita} $\tauxnaissance$. The fact that this latter does not explode is well known and can be derived  for instance from the finiteness of its first moment. Pathwise uniqueness of  the system of stochastic differential equations
is also obtained by induction on the successive jumps, which are provided by the common Poisson point measures. The  argument  above ensure that uniqueness holds for any positive time. Let us finally note that the system (\ref{S1}-\ref{S2}) is closed, since $\mathcal P(t)$ and $(i,r_i(t),a_i(t))$
are determined (uniquely) by the measure $Y(t)$, which is itself determined by its projections $\langle Y(t),f\rangle$ for $f\in \mathcal C^{1,b}(\mathcal U \times \mathcal X)$.
%$\overline{\tauxnaissance}$.
\end{proof}
%\cmv{It also reads for any test function $\varphi$...  on pourra la mettre en appendice la representation trajectorielle, a voir}
%\cmb{tu veux qu on mette l'eds avec des Dirac et des fonctions tests? Je trouve ca clair en l'état}
\iffalse
\begin{proof} Dans l'esprit de la these de Chi.\\
 On cr\'ee des Dirac au moment des "naissances" qu'on ajoute au temps $t$ avec la valeur qu'ils aurait sans mourir ou changer d'etat (juste le vieillissement); quand une mort survient, on supprime  la dirac correspondante au temps $t$. \\
Existence : algorithmique, processus defini par  croissance lineaire des ages entre les evenements de sauts, pas d'explosion pour les sauts car des taux de sauts globaux born\'es
(en fait on commence par dominer le nombre de proies, pour l'instant constant, puis ca borne les taux avec les hypotheses actuelles; plus delicat si les taux peuvent exploser en age).\\
Unicit\'e trajectorielle : idem : sauts identifi\'es par la mesure de Poisson avec une croissance lin\'eaire entre les sauts.
\end{proof}
\fi

\subsection{First estimates and properties}

We start  by a sharp and useful bound on the first moment  of the punctual measures $Y$  evaluated on tests functions which may be non bounded. For convenience, we write

$$\mathcal Y(t,\cdot)= Y(t, \mathcal U,\cdot)= \sum_{i\in \mathcal P(t)} \delta_{(r_i(t),a_i(t))}$$
the projection of the measure $Y(t)$ on $\mathcal X$.
We also introduce the exit time of the number of preys of $ (1/\K,\K)$ for $\K>0$ : 
$$\tau_{\K}=\inf\{ t\geq 0 : X_t \not\in (1/\K,\K)\}. 
$$
We consider the associated bounds on the rate of transitions for $r\in \{S,M\}$,
$$  \overline{\tauxinter}_r(a,\K)=\sup_{x\in (1/\K,\K)} \tauxinter_r(a,x),\quad \underline{\tauxinter}_r(a,\K)=\inf_{x\in (1/\K,\K)} \tauxinter_r(a,x),$$
which are continuous by continuity of $\tauxinter_r$ and
$$
\bar{\tauxnaissance}=\sup_{r\in \{S,M\}, \, a\in[0,a_{\infty})} \tauxnaissance_{r}(a).
$$
 %\qquad \bar{\tauxnaissance}(a)=\sup_{x\in (1/r,r)} \tauxnaissance(a,x)$$
\begin{lem}
\label{lem:boundtightness-pasrenormalise} 
 Under Assumption~\ref{hyp:ci_ws}, there exists 
 $C>0$ such that for any continuous function $f :[0,a_{\infty})\rightarrow \R_+$ and $r\in \{S,M\}$ and $\K>0$,
\begin{align*}
\mathbb{E}\left[\int_0^{T  \wedge \tau_{{\K}}} \int_{[0,a_{\infty})} f(a) \mathcal Y(s,\{r\},da)  ds \right] &\leq C (1+T)e^{\bar{\tauxnaissance} T}   \int_{[0,a_{\infty})} f(a)  e^{-\int_0^a \underline{\tauxinter}_r(u,\K) du/2 } \, da.
 \end{align*}
\end{lem}
%\cmb{J'ai virer $y_0$ car pas introduit et la dépendance en la cond initiale est + complexe, on majore E[exp (alpha)]}
%\cmv{Verifier que $C$ ne depend pas de $\K$}
\begin{proof}
Fix $T\geq 0$ and 
 consider an increasing sequence $a_n$, where $a_{n+1}=a_n+t_n$
 and $(t_n)_n$ is a decreasing sequence of positive numbers fixed hereafter. For a predator $i\in \mathcal U$, a status $r\in \{R,M\}$ and a level $n\in\N$, we set
$$
\pfff^{i,r}_{n} = \mathbb{E}\left[\int_0^{\tau_{\K}\wedge T} \mathbf{1}_{\{i\in \mathcal P(s), \, r_i(s)=r, \, a_i(s)\in [a_n, a_{n+1})\}} ds\right].
$$
 It is equal to the cumulative time spent by  predator $i$, in status $r$ and between ages  $a_n$ and $a_{n+1}$. Let also
$$
N_{n}^{i,r}=\sum_{s\leq \tau_{\K}\wedge T} \mathbf{1}_{\{i\in \mathcal P(s), \, r_i(s)=r, \, a_i(s) =a_n\}}
$$
be the number of times that predator $i\in \mathcal U$ reaches age $a_n$ while it is in status $r$. In other words, writing $b(i)$  the birth time of individual $i$,  for every $j\in \mathbb{N}$, we can  define iteratively
$$T_{j+1,n}^{i,r} = \inf\{t > T_{j,n}^{i,r}\ | \  i\in \mathcal P(t), \, r_i(t)=r,  \, a_i(t)=a_n\},$$ 
for $j\geq 0$, with $T_{0,n}^{i,r}= b(i)$.
We get  $$
N_{n}^{i,r}= \sum_{j\geq 1} \mathbf{1}_{\{T_{j,n}^{i,r} \leq \tau_{\K}\wedge T\}}.$$ With this notation and writing $T_j=T_{j,n+1}^{i,r}$ for convenience, we have 
\begin{align}
\pfff^{i,r}_{n+1}
&\leq \mathbb{E}\left[\sum_{j=1}^{N_{n+1}^{i,r}} \int_{T_{j}}^{T_{j}+t_{n+1}} \mathbf{1}_{\{i\in \mathcal P(s), \, a_i(s)\in [a_{n+1},a_{n+2}),\,  r_i(s)=r\}}  \mathbf{1}_{\{\forall u\in [0,s]: X(u) \in {\K}\}} ds\right].
\end{align}
%\cmv{ATTENTION Es tu d'accord avec cette etape  Adding ajoutee qui me semble justifier le calcul. En realite je crois qu'il faut supposer $t_n$ decroissant au sens large, sinon on pourrait en fait avoir eu le temps d'avoir deux changements de status et revenir dans la tranche voulue il me semble !!! bon ca arrive apres mais si tu valides, je pense qu'il faut le dire des maintenant. Par ailleurs une certaine logique voudrait qu'on definisse plutot
%$ \underline{\tauxinter}_{n}^r=\inf \{ \tauxinter_r(a,x) : \, a\in [a_{n},a_{n+1}], x\in (1/ \K,\K)\}$ et qu'on decalle les indices en fonction}
%\cmb{Je suis ok ac toi sur tt mais la notation actuelle est concise}
Adding that $(t_n)_n$ decreases, the status cannot change twice during time $t_{n+1}$
and come back at level $a_{n+1}$. So  the process does not change at all during this time and for $s\leq t_{n+1}$,
\begin{align*}
 &\mathbb{E}\left[ \mathbf{1}_{\{
 \,  i\in \mathcal P(T_j+s), \, a_i(T_j+s)\in [a_{n+1},a_{n+2}),\,  r_i(T_j+s)=r, \, \forall u \in [0,s]: X(T_j+u)\in \K\} } \, \big\vert \, T_j, \, (X(t))_{t\leq T_j+s}
\right]\\
%&\P(i\in \mathcal P(t+s), \, a_i(s)\in [a_{n+1},a_{n+2}), \, r_i(t+s)=r  \vert  i\in \mathcal P(t), \, r_i(t)=r, \, a_i(t)=a_{n+1})\\
%&\qquad \qquad 
&\qquad \qquad = \mathbb{E}\left[ e^{-\int_0^s \alpha_r(a_{n+1}+u, X(T_j+u))du} 
 \mathbf{1}_{\{\forall u \in [0,s]:
  X(T_j+u)\in \K\}} \big\vert \, T_j, \, (X(T_j+u))_{u\leq s}
\right]\\
&\qquad \qquad \leq  e^{-\int_0^s \underline{\alpha}_r(a_{n+1}+u, \K)du},
\end{align*}
we get 
\begin{align}
\pfff^{i,r}_{n+1}
&\leq \mathbb{E}\left[\sum_{j=1}^{N_{n+1}^{i,r}} \int_0^{t_{n+1}} e^{-\int_0^s \underline{\alpha}_r(a_{n+1}+u, \K)du} \, ds   \right]
\leq t_n \,  p^{r}_{n} \, \mathbb{E}\left[ N_{n+1}^{i,r} \right],
%\times \left( \frac{\tauxnaissance^{i,R}_{n}}{t_n} + \mathbb{P}(a_i(0) \in (n,(n+1)])\right),\label{eq:iterationgammab},
\end{align}
where 
$$p^{r}_{n}=  \frac{1-e^{-\underline{\tauxinter}_{n}^rt_{n+1}}}{\underline{\tauxinter}_{n}^r t_n}, \quad \underline{\tauxinter}_{n}^r=\inf \{ \tauxinter_r(a,x) : \, a\in [a_{n+1},a_{n+2}], x\in (1/ \K,\K)\}.$$ Besides, as ages increase at rate $1$, either predator $i$ is born at an age between $a_{n+1}$ and $a_{n+2}$ or  it has exactly spent the time $t_n$ at level between ages $[a_n,a_{n+1})$. In any case,  
$$
\mathbb{E}\left[ N_{n+1}^{i,r} \right] \leq \P(A_n^i)+
\frac{\pfff^{i,r}_{n}}{t_n},
% \mathbb{P}(T_{1,N}\leq T, \, a_i(T_1^n) \in (a_n,a_{n+1}], i\in \mathcal P_R)
$$
where
$$A_n^{i}=\{b(i)\leq T\wedge \tau_K, \, a_i(b(i)) \in [a_{n+1},a_{n+2})\}.$$
%\frac{1-e^{-\underline{\tauxinter}_{r,n}t_{n+1}}}{\underline{\tauxinter}_{r,n}} .
Combining these inequalities, we obtain
$$\pfff^{i,r}_{n+1}\leq p^{r}_{n}\pfff^{i,r}_{n}+ t_n p^{r}_{n}\P(A_n^i), $$
which then gives, by induction, 
\begin{align*}
\pfff_{n}^{i,r} 
&\leq \pfff_{0}^{i,r} \prod_{j=0}^{n-1} p^r_{j} + \sum_{k=0}^{n-1}   t_k\P(A_k^i) \prod_{j=k}^{n-1} p^r_{j} .
%&\leq \tauxnaissance^{m}_i p^{n-m} + \sum_{k=m}^n  p^n \mathbb{E}[p^{-a_i(0)} \mathbf{1}_{a_i(0) \in (k, k+1]} ]\\
%&\leq T p^{n-m} + p^n \mathbb{E}[p^{-a_i(0)} ].
\end{align*}
Using now $p^r_j \leq \frac{t_{j+1}}{t_j} \left(1 - \frac{\underline{\tauxinter}_{j}^r t_{j+1}}{2} \right)$ and setting
$$S_k^n=  \sum_{j=k}^{n-1} \underline{\tauxinter}_{j}^r t_{j+1}/2, $$
we get $\prod_{j=k}^{n-1} p^r_{j} \leq \frac{t_{n}}{t_k} e^{-S_k^n}$ and then 
\begin{align*}
\pfff_{n}^{i,r} 
&\leq \frac{t_{n}}{t_0} \pfff_{0}^{i,r} e^{-  S_0^n } + t_n \sum_{k=0}^{n-1}   \P(A_k^i)  e^{ -S_k^n} .
%&\leq \tauxnaissance^{m}_i p^{n-m} + \sum_{k=m}^n  p^n \mathbb{E}[p^{-a_i(0)} \mathbf{1}_{a_i(0) \in (k, k+1]} ]\\
%&\leq T p^{n-m} + p^n \mathbb{E}[p^{-a_i(0)} ].
\end{align*}
To conclude, for any continuous function $f$, we set $f_n=\sup_{a\in[a_{n},a_{n+1})} f(a)$ to have
\begin{align*}
&\mathbb{E}\left[\int_0^t  f(a_i(s)) \right.  \left.  \mathbf{1}_{i\in \mathcal P(s), r_i(s)=r}ds\right]\nonumber\\
&\quad \leq   \sum_{n \geq 0} f_n \pfff_{n}^{i,r}\nonumber\\
&\quad\leq  \frac{\P(b(i)\leq T\wedge \tau_\K)}{t_0\wedge 1}
\sum_{n \geq 0} f_n t_n e^{- S_0^n }  \left( 
T +
 \sum_{k=0}^{n-1}  \P(A_k^i \ \vert \  b(i)\leq T\wedge \tau_\K) e^{S_0^k }\right),
\end{align*}
since $\pfff_{0}^{i,r}\leq T \, \P(b(i)\leq T\wedge \tau_\K)$. Choose now $t_0<a_\infty$ be a fixed constant and for $n\geq 1$, $t_n=t_n^h$ any positive sequence, depending on a parameter $h$, in such a way
$$
\lim_{h\to 0} \sup_{n\geq 1} t_n^h=0, \quad \lim_{n \to \infty} a_n^h =  \lim_{n \to \infty} \sum_{k=0}^n t_k^h = a_\infty.
$$
For instance, we can choose $t_n^h = h$ when $a_\infty=+\infty$. %\cmv{Attention l'argument me semble peu clair, je doute de l'uniformite sur l'espace complet, j'aurais pense plus a un argument de convergence mononone jouant sur les inf...}   
Using the convergence of the Darboux sum
$S_0^k$ to  $\int_{[0,a)} \underline{\tauxinter}_r (u) du/2$ when $a^h_k\rightarrow a$, which comes from the continuity of $\underline{\tauxinter}$ and $f$, we get, by letting $h\to 0$,

 \begin{align*}
&\mathbb{E}\left[\int_0^{\tau_{\K}\wedge T}  f(a_i(s))  1_{\{i\in \mathcal P(s), \, r_i(s)=r\}}ds\right]\\
&\qquad \leq C \mathbb{P}(b(i)\leq T\wedge \tau_{\K}) \int_{[0,a_\infty)} 
f(a)  e^{-\int_{[0,a)} \underline{\tauxinter}_r (u,\K) du/2 } da \\
&\qquad \qquad  \qquad \qquad  \qquad \qquad \times \left[T +  
 \E\left( {\bf 1}_{\{a_i(b(i))\leq a\}} \,\exp\left(\int_0^{a_i(b(i))}  \underline{\tauxinter}_r (u,\K) du/2  \right) \right) \right].
\end{align*}%\cmb{Ca te va? cf mon autre fichier}
More precisely, the previous inequality is obtained using uniform continuity when $f$ is compactly supported over $[0,a_\infty)$ and  extended for every function by a truncation argument and Fatou Lemma.\\

Recall that  $a_i(b(i))$ has a compact support in $[0,a_{\infty})$ by Assumption~\ref{hyp:ci_ws} and the fact that newborns have age $0$. So the last term is bounded by a constant.
Summing over all predators $i$ yields the result since
$$
\sum_{i\in \mathcal U}  \mathbb{P}(b(i)\leq T\wedge \tau_{\K}) \leq 
\E\left( \#\{ i\in \mathcal U : b(i)\leq T\wedge \tau_{\K}\}\right)\leq  e^{T \, \overline{\tauxnaissance}} \, \mathbb{E}\left[ \mathcal{Y} (0, \{S,M\}, [0,a_\infty)) \right] %\E(\mathfrak Z(T\wedge \tau_{\K})) ,
$$
since $\#\{ i\in \mathcal U : b(i)\leq T\wedge \tau_{\K}\}$ is dominated by a pure birth
process at time $T$, with individual  birth rate $\overline{\tauxnaissance}$.
%where $\mathfrak Z(t)$ is the process counting the number of predators born before time $t$, which is bounded by a linear birth process  so expectation grows at most at rate $\sup_{\K} \tauxnaissance(r,.)$.
\end{proof}

We define $F_{g,f}: \mathbb{R}_+ \times \mathfrak{M}(\mathcal X) \rightarrow \R$ by
\begin{equation}
\label{eq:F}
F_{g,f}(x,\mu)=g(x)+\langle\mu,f\rangle,
\end{equation}
where   $g : \R_+\rightarrow \R$ and $f : \mathcal{X} \rightarrow \R$ are measurable and bounded functions. We introduce
 \Bea
 \mathfrak LF_{g,f}(x,\mu)
& =& \tauxnaissance x(g(x+1)-g(x))+\tauxmort x(g(x-1)-g(x))\\
&&+\int_{\mathcal X} \left(\frac{\partial }{\partial a} f(r,a)+\tauxnaissance_r(a) f(M,0) -
\tauxmort_r(a)f(r,a)\right)\mu(dr,da)\\
&& +\int_{\mathcal X} \tauxinter_r(a,x)(1_{r=S}\left(g(x-1)-g(x)\right)+f(\overline{r},0)-f(r,a)) \mu(dr,da).
%&&\qquad +\int_{[0,a_\infty)} \tauxinter(M,a,x)(f(S,0)-f(M,a))\mu_M(da)
%&&\qquad +\int_{[0,a_\infty)} \sum_{R\in \{S,M\}} \left(\right)\mu_R(da)\\
\Eea
The operator $\mathfrak L$ is the generator of the Markov process $(X(t), \mathcal Y(t))_{t\geq 0}$. % in the sense of \cite{MTIII} or \cite{D93}, or $(Z(t))_{t\geq 0}$ is solution of the martigale problem for $L$ in the sense of \cite{EK09} (ON A SUBDOMAIN here functions as  sum of 3 components). 
More precisely, our SDE  representation (\ref{S1}-\ref{S2}) ensures the following classical martingale 
problem.

\begin{lem} 
\label{lem:mart-sansscaling}  Assume that Assumption~\ref{hyp:ci_ws} holds and that for any $\K>0$ and $r\in \{S,M\}$,
\begin{equation}
\label{hyp:alpha_ws}
\int_{[0,a_{\infty})} \overline{\tauxinter}_r(a,\K)  e^{-\int_0^a \underline{\tauxinter}_r(u,\K) du/2 }da <\infty.
\end{equation}
Let  $g : \R_+\rightarrow \R$ be measurable and bounded and $f \in \mathcal C^{1,b}( \mathcal{X})$. Then  $(M(t))_{t\geq 0}$ defined for $t\geq 0$  by
$$
 M(t) = F_{g,f}(X(t),\mathcal Y(t)) - F_{g,f}(X(0),\mathcal Y(0)) - \int_0^t \mathfrak L F_{g,f}(X(s),\mathcal Y(s)) ds
$$
is  a local martingale. Besides $\left( M(t\wedge\tau_\K)\right)_{t\geq 0}$ is a square-integrable martingale and its bracket is given, for all $t\geq 0$, by 
\Bea
&&\langle M \rangle(t\wedge\tau_\K)\\
&&\quad=\int_0^{t\wedge\tau_\K}  \left(\tauxnaissance X(s) (g(X(s)+1)-g(X(s)))^2+\tauxmort X(s)(g(X(s)-1)-g(X(s)))^2\right)ds\\
&&\qquad +\int_0^{t\wedge\tau_\K}  \sum_{i \in \mathcal P(s)} \tauxinteri(s)\left(1_{r_i(s)=S}(g(X(s)-1)-g(X(s)))+(f(\overline{r_i}(s),0)-f(r_i(s),a_i(s))\right)^2ds\\
%&&+\int_0^{t}  \sum_{i \in \mathcal P_M(s)} \tauxinter_M(a_i(s-),X(s-))\left(f_S(0)-f_M(a_i(s))\right)^2ds\\
&&\qquad +\int_0^{t\wedge\tau_\K}  \sum_{i \in \mathcal P(s)} \left(\tauxnaissancei(s)f(M,0)^2+\tauxmorti(s) f(r_i(s),a_i(s))^2\right) ds.
%+\int_0^{t}  \sum_{i \in \mathcal P_M(s)}  \tauxmort_M(a_i(s))f_M(a_i(s-))^2 ds
\Eea
\end{lem}
\begin{proof} The fact that $M$ is a local martingale and the computation of its square variation 
 is derived from our SDE representation  (\ref{S1}-\ref{S2}) . Indeed one can 
write the semi-martingale decomposition of $F_{g,f}(X,\mathcal Y)$ using the Poisson point measure and its compensation, see   \cite{IW14}. We only give details for the first component $X$:
 \begin{align*}
g(X(t))&= g(X(0))
+\int_0^t  \sum_{\substack{i\in \mathcal P(s-) \\ r_i(s-)=S}}   \tauxinteri(s-)(g(X(s-)-1)-g(X(s-)) \, ds \nonumber  \\
&\quad +\int_0^t  \sum_{\substack{i\in \mathcal P(s-) \\ r_i(s-)=S}}  \int_{\mathbb{R}_+}\mathbf{1}_{\{u\leq \tauxinteri(s-)\}} (g(X(s-)-1)-g(X(s-)) \, \widetilde{\mathcal N}^i(ds,  du)\nonumber \\
& \quad +\int_0^t  \left(\tauxnaissance X(s-)((g(X(s-)+1)-g(X(s-)) +
 \tauxmort X(s-) (g(X(s-)-1)-g(X(s-))\right)\, ds\\
&\quad +\int_0^t \int_{\mathbb{R}_+}  \Big(\mathbf{1}_{\{u\leq \tauxnaissance X(s-)\}}((g(X(s-)+1)-g(X(s-)) \\
& \qquad \qquad \qquad \qquad     +\mathbf{1}_{\{0 < u- \tauxnaissance X(s-)\leq \tauxmort X(s-)\}} (g(X(s-)-1)-g(X(s-))\Big)\, \widetilde{\mathcal Q}(ds,  du), \label{S1}
\end{align*}
where $\widetilde{\mathcal N}^i$ and  $\widetilde{\mathcal Q}$ are the compensated measures of $\mathcal N^i$ and $\mathcal Q$. 

Finally,  square integrability of $(M(t\wedge\tau_\K))_{t\geq 0}$ is a 
consequence of Lemma~\ref{lem:boundtightness-pasrenormalise} applied to
$f=\overline{\alpha}_r$ and Doob's inequality and our integrability assumption \eqref{hyp:alpha_ws}. 
\end{proof}
\section{Scaling and averaging}
%\section{Premiere limite d' Êechelle : averaging sans naissance mort autre que la predation}
\label{se:scaling}

\subsection{Approximation of the scaled process by a dynamical system}

%\cmv{peut etre $\mathfrak{s}$ et $\mathfrak{h}$ pour les statuts ? plutot que $S,M$ ?}
Let us now introduce our scaling parameters $K=(K_1, K_2) \in (0,+\infty)^2$  respectively for the size of the populations of  preys and the predators. These sizes are going to infinity. Besides, in our scaling, 
$$\lambda_K=\frac{K_1}{K_2}$$
tends to infinity.  The intial number of preys and predators satisfy
$$X^K(0)= \lfloor K_1 x_0\rfloor, \qquad  Y^K( 0,\mathcal U, \{S,M\}, [0,a_\infty))=\lfloor K_2 y_0\rfloor,$$ for some constants $x_0,y_0>0$. 
The rates for interactions are now density-dependent (rather that population-size-dependent) and we set
$$
\tauxinter_r^K(a,x)= \tauxinter_r(a,x/K_1)%, \qquad \tauxinter_M^K(a,x)= \tauxinter_M(a,x/K_1),
$$
for $r\in \{S,M\}$, where $\tauxinter_r$ is measurable on $[0,a_{\infty})\times \R_+$. Interactions occur at a fast time scale which arises here  through an acceleration of time. Birth and death of preys and predators (but the deaths of preys due to predation) occur at a slower time scale and we set

$$\tauxmort_r^K( a) =\lambda_K^{-1}\tauxmort_r(a), \quad \tauxnaissance_r^{K}(a)=\lambda_K^{-1}\tauxnaissance_r(a),  \qquad  \tauxmort^K=\lambda_K^{-1}\tauxmort, \quad \tauxnaissance^{K}=\lambda_K^{-1}\tauxnaissance,$$
%$$\tauxmort_S^K : a \mapsto \lambda_K^{-1}\tauxmort_S(a), \qquad \tauxmort_M^{K}=\lambda_K^{-1}\tauxmort_M.$$
where  $\tauxmort_r^K$ and  $\tauxnaissance_r^K$  are non-negative, measurable and bounded functions and $\tauxmort,\tauxnaissance$ are non-negative numbers. See Section~\ref{sect:discussion} for a discussion on our scaling.
\begin{assumption}
\label{hyp:ci}
There exists $a_0 \in (0,a_\infty)$ such that $\mathcal Y^K(0,\{S,M\},[a_0,\infty))=0$ a.s. for all $K\geq 1$. 
%{hyp:ci_ws} There exists $a_0 \in (0,a_\infty)$ such that $ Y(0,\{S,M\},[a_0,\infty))=0$ a.s. for all $K\geq 1$. 
Besides, for any $r\in \{S,M\}$ and $\K>0$,
$$\inf_{x\in [0, \K]}\E(T_r(x)) >0, \qquad \sup_{a\in[0,a_{\infty})} \tauxnaissance_{r}(a)+\tauxmort_{r}(a) <\infty $$
and $a\rightarrow \alpha_r(a,x)$ is continuous on $[0,a_{\infty})$ for any $x\in \R_+$.
\end{assumption}
Under this assumption, for each $K=(K_1,K_2)\in (0,+\infty)^2$, Proposition~\ref{prop:existe} ensures existence and strong uniqueness of the solution $Z^K=(X^K, Y^K)$ of the system of stochastic differential equations (\ref{S1}-\ref{S2}) with parameters 
$\tauxinter_r^K,\tauxnaissance^K_r,\tauxmort^K_r,\tauxnaissance^K,\tauxmort^K$ given above. We consider the accelerated and scaled process defined, for all $t\geq 0$, by
$$\mathcal Z^K(t)=(\Xi^K (t),\mathcal Y^K(t,dr,da))=\left(\frac{1}{K_1}X^K(\lambda_Kt), \frac{1}{K_2}Y^K(\lambda_K t,\mathcal U, dr,da),  \right).$$
For every $T>0$, Process $(\mathcal Z^K(t), t\in [0,T])$ belongs to the space $\mathbb{D}([0,T], \mathbb{R}_+) \times \mathfrak{M}([0,T] \times \mathcal X)$. Space $\mathbb{D}([0,T], \mathbb{R}_+)$ is the classical Skorokhod space with its usual topology; see for instance \cite{B13}. Space $\mathfrak{M}([0,T] \times \mathcal X)$ is the space of finite positive measures on $[0,T] \times \{S,M\} \times [0,a_\infty)$ embedded with  narrow convergence.\\

Our result relies  on the following assumption on the interaction rates. It is slightly  stronger than \eqref{hyp:alpha_ws} and is involved in tightness proof to localize age in compact sets of $[0,a_{\infty})$.\\
Analogously to the un-scaled setting, we set for $r\in \{S,M\}$ and $\K>0$,
$$\underline{\tauxinter}_r(a,\K)=\inf_{x\in (-1/\K,\K)} \tauxinter_r(a,x), \quad \overline{\tauxinter}_r(a,\K)=\sup_{x\in (-1/\K,\K)} \tauxinter_r(a,x), $$
and 
$$\bar{\tauxnaissance}=\sup_{r\in \{S,M\}, \, a\in[0,a_{\infty})} \tauxnaissance_{r}(a).$$

\begin{assumption} 
For any $\K>0$, there exists a continuous function $\lyap : [0,a_{\infty})\rightarrow [1,\infty)$ such that 
$
\lim_{a\rightarrow a_{\infty}} \lyap(a) =+\infty$
%, \qquad  \sup_{a\in [0,a_{\infty})} \frac{\overline{\tauxinter}_S(a,\K)}{\lyap(a)}<\infty$$
 and for $r\in \{S,M\}$,
\label{hyp:alpha}
$$\int_ {[0,a_{\infty})}  \,  \lyap(a) \, (1+\overline{\tauxinter}_r(a,\K)) e^{-\int_{[0,a)}  \underline{\tauxinter}_r(s,\K) \, ds /2} da <\infty.$$
\end{assumption}
For $r\in \{S,M\}$, we write for convience
\begin{equation}
\label{eq:pr}
p_r(x,a)=e^{-\int_0^a \tauxinter_r(x,u)du}
\end{equation}
 the cumulative distribution of the interaction times.  We define
\begin{equation}
\label{eq:phi}
\phi(x)=  \frac{1}{\int_{[0,a_{\infty})} (p_S(x,a) + p_M(x,a))  da}
\end{equation}
and
\begin{equation}
\label{eq:psi}
\psi(x)=\phi(x)\, \int_{[0,a_{\infty})}\left(\tauxmort_S(a){p_S(x,a)}+\tauxmort_M(a)p_M(x,a)  \right)da.
\end{equation}
Let us refer to Equation~\eqref{eq:phi-intro} and \eqref{eq:psi-intro} in introduction for an expression of $\phi$ and $\psi$ in terms of the random variables $T_{r}$ and the demographic rates. 
Our last assumption concerns uniqueness of the limiting equation and the fact that the limit does not reach a boundary. For simplicity, we also assume here existence, but the limiting procedure we prove would ensure existence up to this time when the process get close to the boundary.  %However, uniqueness is, as usual, not guaranteedby the continuity of our parameters.

\begin{assumption}
\label{eq:systeme-edo}
The following system of ordinary differential equations,

\begin{equation}
\label{ODElimit}
\left\{
\begin{array}{rcl}
x'(t)&=&(\gamma-\beta) x(t)-y(t)\phi(x(t)),\\
y'(t)&=&y(t)\psi(x(t)),
\end{array}
\right.
\end{equation}
admits a unique global solution $(x,y) \in \mathcal C^1(\R_+, (\mathbb{R}_+^*)^2)$ such that $(x(0),y(0))=(x_0,y_0)$.
\end{assumption}
The preceding assumption holds under classical regularity
assumption and in particular, if $\phi$ and $\psi$ are globally Lipschitz. Locally Lipschitz conditions are also sufficient when the system does not explode in finite time. That is enough for our purpose.
Our main result can be stated as follows.

\begin{thm}
\label{thm:main} 
Under Assumptions~\ref{hyp:ci},~\ref{hyp:alpha} and ~\ref{eq:systeme-edo}, for every $T>0$, the two following assertions hold : 
\begin{itemize}
\item[i)] the process $(\Xi^K(t),\mathcal Y^K(t,\{S,M\},[0,a_\infty))_{t\in [0,T]}$ converges in law to  $(x(t),y(t))_{t\in [0,T]}$ in $\mathbb D([0,T],\R_+^2)$,%\item Process $\mapsto \mathcal Y^K_S(\cdot, \{S,M\}, \R_+)$ converges in law to $y$ in the space of measurable functions embedded with the  weak * topology on $C_b$
\item[ii)] for each  $r\in \{S,M\}$, the measure $\mathcal Y^K(t, \{r\}, da)dt$ converges in law  to the measure 
$$y_r(dt,da)=y(t)p_r(x(t),a) \phi(x(t)) \, dt\, da$$
 in  the space $\mathfrak{M}([0,T]\times [0,a_{\infty}))$.
\end{itemize}
\end{thm}
%In particular, the total number of  predators converges to $y$ in the following  weak sense
%. For all continuous and bounded functions $F,f$ we have
%$$
%\lim_{K \to \infty } \P\left( \bigg\vert \int_0^T\int_{[0,a_{\infty})} f(t) \mathcal Y^K(t, \{S,M\}, da) dt  - \int_0^T f(t) y(t) dt \bigg\vert \geq \varepsilon \right)=0,
%$$
%for any $T>0$ and $\varepsilon >0$ and $f$ continuous and bounded on $[0,T]$.
%Note that, we can write, for every $x>0$,
%$$
%\phi(x)=  \frac{1}{\E[T_S(x)]+\E[T_M(x)]},\quad 
%\psi(x)=\frac{\E\left[ \int_0^{T_S(x)} \tauxmort_S(u) du\right]+\E\left[ \int_0^{T_M(x)} \tauxmort_M(u) du\right]}{\E[T_S(x)]+\E[T_M(x)]}.
%$$
%Indeed, our convergence result holds for the mathematical object $\mathcal Y^K(t,dr,da)dt$ which is the natural quantity captured by the averaging results. 
The fact that convergence of $\mathcal Y^K(t,dr,da)$  hold on the associated Skorokhod space is left open.
%or Finally, before to state our limit result, let us describe rapidly the involved typologies. Process $(\Xi^K)$ belongs to the classical $\mathbb D([0,\infty),\R_+)$ embedded with its usual topology \cite{EK09,B13}. We could have considered that $\mathcal Y^K$ belongs to  $\mathbb D([0,\infty),\mathfrak{M})$. However, due to the averaging limit, this process seems to do not converge for this topology and we will then use a weaker topology. We identify $\mathcal Y^K$ by the measure on $\overline{\mathfrak{M}}_T$ the set of finite measure on $[0,T]\times \mathcal{X}$ with the classical weak* convergence.

\subsection{Proofs}

The proof is based on standard tightness and uniqueness arguments involving the occupation measures and averaging \cite{K92,KKP} and localization. The main new difficulties lie in the infinite dimension in the averaging procedure due to the age structure 
combined with unboundedness of the interactions rates $\tauxinter_{r}$ inherent in our framework. \\
%In infinite dimension, we know only one work \cite{MT12} involving such averaging. It provides also an interesting case of age structure. Here the age describes interaction and the main mathematical new difficulty arising is the unboundedness of the rates.

First, 
Lemma~\ref{lem:boundtightness-pasrenormalise} above directly implies the following counterpart for the scaled process. It allows us to localize the age distribution. We set
$$\tau_{\K}^K=\inf\{ t\geq 0 : \Xi_t^K \not\in (1/\K,\K)\}.$$
\begin{lem}
\label{lem:boundtightness-renormalise}
Under Assumption~\ref{hyp:ci}, there exists $C>0$ such that for any continuous function $f$
on $[0,a_{\infty})$  and $r\in \{S,M\}$ and $\K>0$,
\begin{align*}
\mathbb{E}\left[\int_0^{T  \wedge \tau_{{\K}}^K} \int_{[0,a_{\infty})} f(a) \mathcal Y^K(s,\{r\},da)  ds \right] &\leq C (1+T)e^{\bar{\tauxnaissance}_r T}   \int_{[0,a_{\infty})} f(a)  e^{-\int_{[0,a)} \underline{\tauxinter}_r(u,\K) du/2 } da.
 \end{align*}
\end{lem}

\begin{proof} 
We have
\begin{align*}
&\mathbb{E}\left[\int_0^{T \wedge  \tau^K_{\K}}\int_{[0,a_{\infty})} f(a) \mathcal Y^K(s,\{r\},da)  ds \right] \\
 &\qquad \qquad = \frac{1}{\lambda_K K_2}\mathbb{E}\left[\int_0^{\lambda_K(T \wedge\tau_ {\K}^K)}\int_{[0,a_{\infty})} f(a)  Y^K(s,\{r\},da)  ds  \right].
%&\qquad \leq \int_{0}^{a_\infty} f_S(a)  e^{-\int_0^a \underline{\tauxinter}_S (u) du } \mathbb{E}\left[ \int_{\mathbb{R}_+} \left(T+ \frac{1}{\lambda_K}e^{\int_0^{a} \underline{\tauxinter} _S(u) du} \right) \frac{Y^K_S(0,da)}{K_1} \right]\\
%&\qquad \qquad + \int_{0}^{a_\infty} f_M(a)  e^{-\int_0^a \underline{\tauxinter}_M (u) du } \mathbb{E}\left[ \int_{\mathbb{R}_+} \left(T+ \frac{1}{\lambda_K}e^{\int_0^{a} \underline{\tauxinter}_M (u) du} \right) \frac{Y^K_M(0,da)}{K_1}, \right]
\end{align*}
Adding that  $\lambda_K \tau_ {\K}^K$ is the exit time of $(K_1/\K, K_1\K)$ for $X^K$ and 
 $Y^K( 0,\mathcal U, \{S,M\}, [0,a_\infty))=\lfloor K_2 y_0\rfloor$ and  $\tauxnaissance_r^{K}(a)=\lambda_K^{-1}\tauxnaissance_r(a)$, the conclusion comes from Lemma~\ref{lem:boundtightness-pasrenormalise}.
\end{proof}

We now give the counterpart of  the martingales  of Lemma~\ref{lem:mart-sansscaling} for the   scaled process.  
Recalling that $
F_{g,f}(x,\mu)=g(x)+\langle\mu,f\rangle$ where   $g : \R_+\rightarrow \R$ is a bounded measurable function and $f\in \mathcal C^{1,b}( \mathcal{X})$, we set
\Bea
&&\mathcal L^KF_{f,g}(x,\mu)= K_1 x\big(\tauxnaissance(g(x+1/K_1)-g(x))+\tauxmort(g(x-1/K_1)-g(x))\big)\\
&&\quad +\int_{\mathcal X} \left(\tauxnaissance_r(a) f(M,0) -
\tauxmort_r(a)f(r,a)\right)\mu(dr,da)\\
&& \quad +\lambda_K \int_{\mathcal X}\left( \frac{\partial }{\partial a} f(r,a)+ \tauxinter_r(a,x)\left(1_{r=S}\left(g(x-1/K_1)-g(x)\right)+f(\overline{r},0)-f(r,a)\right)\right) \mu(dr,da).
%&&\qquad +\int_{[0,a_\infty)} \tauxinter(M,a,x)(f(S,0)-f(M,a))\mu_M(da)
%&&\qquad +\int_{[0,a_\infty)} \sum_{R\in \{S,M\}} \left(\right)\mu_R(da)\\
\Eea

\begin{lem}
\label{lem:mart-avecscaling} Suppose that Assumptions \ref{hyp:ci} and   \ref{hyp:alpha} hold. Let  $g : \R_+\rightarrow \R$  be a bounded measurable function and $f\in \mathcal C^{1,b} ( \mathcal{X})$. Then the process $M^K$ defined for $t\geq 0$ by
$$
M^K(t) = F_{f,g}(\mathcal Z^K(t)) - F_{f,g}(\mathcal Z^K(0)) - \int_0^t \mathcal L^K F_{f,g}(\mathcal Z^K(s)) ds,
$$
is a local martingale. Besides $(M^K(t\wedge \tau_{\K}^K))_{t\geq 0}$
is a square integrable martingale
% If moreover $f$ is bounded and there exists $C>0$ such that$$ \forall x\geq 0, \qquad |g(x)| \leq C(1+x),  $$ then $(M^K(t))_{t\geq 0}$ is a square-integrable martingale and 
and
\Bea
&&\langle M^K\rangle(t\wedge \tau_{\K}^K) \\
&&\quad =\int_0^{t\wedge \tau_{\K}^K}  \Xi^K(s)\left(\tauxnaissance(g(\Xi^K(s)+1/K_1)-g(x))^2+\tauxmort(g(\Xi^K(s)-1/K_1)-g(x))^2\right)ds\\
&&\qquad \, +\lambda_K\int_0^{t\wedge \tau_{\K}^K}  \sum_{i \in \mathcal P(s)} \tauxinter_{r_i(s)}(a_i(s),\Xi^K(s))\\
&&\qquad \qquad  \times \left(1_{r_i(s)=S}(g(\Xi^K(s)-1/K_1)-g(\Xi^K(s)))+\frac{1}{K_2}(f(\overline{r_i}(s),0)-f(r_i(s),a_i(s))\right)^2ds\\
%&&+\int_0^{t}  \sum_{i \in \mathcal P_M(s)} \tauxinter_M(a_i(s-),X(s-))\left(f_S(0)-f_M(a_i(s))\right)^2ds\\
&&\qquad \, +\int_0^{t\wedge \tau_{\K}^K}  \sum_{i \in \mathcal P(s)} \left(\tauxnaissance_{r_i(s)}(a_i(s))\frac{f(M,0)^2}{K_2^2}+\tauxmort_{r_i(s)}(a_i(s)) \frac{f(r_i(s),a_i(s))^2}{K_2^2}\right) ds.
%+\int_0^{t}  \sum_{i \in \mathcal P_M(s)}  \tauxmort_M(a_i(s))f_M(a_i(s-))^2 ds
\Eea
\end{lem}

We  introduce now the measures  $\Gamma^{K}_{\K}$   on  $\R_+\times \{S,M\}\times [0,a_{\infty})$ defined a.s. for every bounded measurable functions $H$ by
$$
\Gamma^{K}_{\K}(H)= \int_{\mathbb{R}_+}\int_{\mathcal X} H(s,r,a) \Gamma^{K}_{\K}(ds,dr, da)  =\int_0^{\tau_\K^K} \int_{\mathcal X} H(s,r,a) \mathcal Y^K(s,dr,da)ds
$$
%$$ \tauxnaissance^{r,K}_M(H)= \int_{\mathbb{R}_+^2} H(s,a) \tauxnaissance^{r,K}_M (ds,da) = \int_0^{\tau_r^K} \int_{\mathbb{R}_+} H(s,a) \mathcal Y_M^K(s,da)ds.$$
We also set $$\Xi^{K}_{\K}(t)=\Xi^K( t\wedge \tau_\K^K), \qquad
\mathcal Y^K_{\K}(t)=\mathcal Y^K(t\wedge \tau_{\K}^K ,\{S,M\},[0,a_\infty))$$
for the localized version of the processes counting preys and predators.
 Considering such space-time measures for proving averaging  results is inspired from \cite{K92,KKP}. However, we do not consider here the occupation measure of the fast variables $ \mathcal Y^K(t,dr, da)$.

\begin{lem}
\label{lem:tightness}
For every $\K >0$ and $T>0$, the sequence $(\Xi^{K}_{\K},\mathcal Y^K_{\K},\Gamma^{K}_{\K})_K$ is tight in $ \mathbb D([0,T],\R_+)^2\times \mathfrak  M ([0,T]\times \mathcal{X})$. 
\end{lem}

\begin{proof}
On the first hand, using a domination of the process $\mathcal{Y}^K(\cdot ,\{S,M\}, [0,a_{\infty}))$  by a linear birth process, we have
\begin{equation}
\label{dominnb}
\sup_K \E\left( \sup_{t\leq T}  \mathcal{Y}^K( s,\{S,M\}, [0,a_{\infty}))\right)<\infty.
\end{equation}
Then the first moment of $(\Gamma^{K}_{\K}([0,T]\times  \{S,M\}\times [0,a_{\infty})))_{K}$ is  bounded and it is a tight sequence in $\R$.% \cmb{modifié ici et plus bas}

On the second hand, we can combine Assumption~\ref{hyp:alpha} with Lemma~\ref{lem:boundtightness-renormalise} to obtain
\be
\label{borneencore}
\sup_{K\geq 1} \mathbb{E}\left[ \int_0^{\tau^K_{\K} \wedge T}\int_{[0,a_{\infty})} \lyap(a) \mathcal{Y}^K( s,\{r\},da) \, ds  \right] <+\infty.
\ee
%with $\phi_R: a \mapsto a^{(\delta+\epsilon)/2}$ in case \ref{hyp:alpha-infini} and $\phi_R: a \mapsto (a_\infty-a)^{-(1+\epsilon)}$ in case \ref{hyp:alpha-fini}. 
for $r\in \{S,M\}$. Then $\sup_K \Gamma^{K}_{\K}(H)<\infty$, with $H(r,a,s)=\lyap(a)\mathbf 1_{s\leq T}$ tending to infinity as $a\rightarrow a_{\infty}$, uniformly  for $s\leq T$. Lemma 1.1 of \cite{K92} then entails the relative compactness of the sequence $(\Gamma^{K}_{\K})_{K\geq 1}$ in the  space of finite measures embedded with the weak (narrow) topology. %\textcolor{red}{We can also extend this relative compactness to the set $D_{\phi_R}$, set of measure that integrate all function dominated by $\phi_R$, by using \cite[Lemma 2.9]{KKP}} 

Now, we show that $(\Xi^{K}_{\K})_{K\geq 1}$ is tight by using the Aldous-Rebolledo criterion. Lemma~\ref{lem:mart-avecscaling} gives the semi-martingale decomposition $$\Xi^{K}_{\K}=\Xi^{K}_{\K}(0)+A^K +M^K,$$ where
 \begin{align*}
A^K(t)&= \int_0^{t\wedge \tau^K_{\K}} (\tauxnaissance-\tauxmort)\Xi^K_{\K}(s)ds -\int_{0}^{t\wedge \tau^K_{\K}} \int_{[0,a_{\infty})}  \tauxinter_S( a,\Xi^K_{\K}(s)) \mathcal Y^K(s,\{S\},da) ds,  \\
\langle M^K\rangle (t)&= \frac{1}{K_1}\int_0^{t\wedge \tau^K_{\K}} (\tauxnaissance+\tauxmort)\Xi^K_{\K}(s)ds +\frac{1}{K_1}\int_{0}^{t\wedge \tau^K_{\K}} \int_{[0,a_{\infty})} \tauxinter_S(a,\Xi^K_{\K}(v))  \mathcal Y^K(v,\{S\},da)ds.
\end{align*}
Hence, writing $\mathcal T^K$ the set of stopping times associated to $\Xi^{K}_{\K}$,
for any  $\sigma\in \mathcal T^K$  and $h>0$,
\begin{align*}
&\mathbb{E}\left[\vert A^K(\sigma )-A^K(\sigma+h)\vert \right]\\
&\qquad \leq h \,  \K \, \vert\tauxnaissance-\tauxmort\vert + \mathbb{E}\left[\int_{\sigma\wedge \tau^K_{\K}}^{(\sigma+h)\wedge \tau^K_{\K}} \int_{[0,a_{\infty})}\tauxinter_S(a,\Xi^K(v)) \mathcal Y^K(v,S,da) dv \right]. %\leq  C'_\K(t-s).
\end{align*} 
%\cmv{J'ai l'impression que ca ne suffit pas ! il faudrait que $\lyap$ ecrase $\tauxinter_S$ non ? j'ai l'impression que c'est un faux probleme et qu'on peut faire sauter une hypothese dans 3.2, qu'on peut obtenir le lemme pour f non croissate en la supposant juste continue et en faisant tendre le pas vers 0. Ou bien en mettant une borne fini avec une version croissante du taux}
%\cmb{On a pas besoin que $V$ ecrase $\alpha$. Si je pose $F(b,(s,omega)) = \int_{b,a} \alpha $ alors elle est dominé par $g(s,omega)=V(a)$ donc par cv dominé $lim_b E[\int_0^T F(b,..)]]= E[\int lim_b F(b,)]]$ et $lim_b F(b,.)=0$, non?}
Using again  Assumption~\ref{hyp:alpha} and Lemma~\ref{lem:boundtightness-renormalise}
with now $f(a)=\overline{\alpha}_S(a,\K)$, we get
$$ 
\lim_{b\to a_\infty}\sup_{K\geq 1} \mathbb{E}\left[ \int_0^{\tau^K_{\K} \wedge T}\int_{[b,a_{\infty})}\tauxinter_S(a,\Xi^K(v)) \, \mathcal{Y}^K( s,\{S\}, da)  ds \right] = 0.
$$
Using \eqref{dominnb} and that  $\tauxinter_S$ is bounded on compacts sets of $[0,a_{\infty})\times (0,\infty)$ by continuity,  we obtain for any $b\in [0,a_{\infty})$,
$$\lim_{h\rightarrow 0}  \sup_{\substack{K\geq 1,\\   \sigma \in \mathcal T^K, \, h \leq \delta}} \mathbb{E}\left[\int_{\sigma\wedge \tau^K_{\K}}^{(\sigma+h)\wedge \tau^K_{\K}} \int_{[0,b]}\tauxinter_S(a,\Xi^K(v)) \mathcal Y^K(v,S,da) dv \right]=0.$$
Combining these estimates yields
\begin{equation}
\label{eq:accroissement}
 \lim_{\delta \to 0}  \sup_{\substack{K\geq 1,\\   \sigma \in \mathcal T^K, \, h \leq \delta}} 
 \mathbb{E}\left[\vert A^K(\sigma )-A^K(\sigma+h)\vert \right]= 0.
\end{equation}
Proceeding  analogously  for  the quadratic variation of $M^K$ and using \cite[Theorem 2.3.2]{JM86} ends the proof of tightness of $(\Xi^{K}_{\K})_{K\geq 1}$. The proof of tightness of $(\mathcal Y^K_{\K})_{K\geq 1}$ is similar since birth and death rates are bounded.
%Indeed, using Lemma~\ref{lem:mart-avecscaling}, with $g\equiv 0$ and $f \equiv 1$, we can decompose the number of predators as a semi-martingale. Using boundedness of $\tauxgrowth_S,\tauxgrowth_M$ and \eqref{dominnb}, we have the equi-continuity type convergence similar as \eqref{eq:accroissement} (and also for the martingale bracket) and then by \cite[Theorem 2.3.2]{JM86} the last tightness result.
\end{proof}

We proceed now  with identification of limiting points. Recall that the survival function of interaction times  is denoted by $p_r$  in \eqref{eq:pr}
 and response for prey is $\phi$, see \eqref{eq:phi}.

\begin{lem}\label{lem:gamma} Let $T>0$,  $\K_0>0$ and consider a limiting point $(\Xi_{\K_0},\Gamma_{\K_0})$ of $(\Xi^{K}_{\K_0},\Gamma^{K}_{\K_0})$ in $\mathbb{D}([0,T], \mathbb{R}_+) \times \mathfrak{M}([0,T]\times \mathcal X)$. For all but countably many $\K<\K_0$, it satisfies  for any  $r\in \{S,M\}$, and $f$ continuous bounded on $\R_+\times [0,a_{\infty})$, 
\begin{align*}
&\int_0^{\tau_{\K}}\int_{[0,a_{\infty})} f(s,a) \Gamma_{{\K_0}}(ds,\{r\},da)\\
&\qquad = \int_0^{\tau_{\K}}\int_{[0,a_{\infty})} f(s,a)  p_r(\Xi_{\K_0}(s),a) \phi\left(\Xi_{\K_0}(s)\right) \, \Gamma_{\K_0}(ds,\{S,M\}, [0,a_\infty) )da \qquad \text{a.s.},
\end{align*}
 where
$$
\tau_\K= \inf\left\{ t\geq 0 \ | \  \Xi_{\K_0}(t) \notin (1/\K,\K)  \right\}.
$$
\end{lem}

\begin{proof}
To avoid the use of a sub-sequence, we assume that the sequence $(\Xi^{K}_{\K_0},\Gamma^{K}_{\K_0})_K$ converges in law to $(\Xi_{\K_0},\Gamma_{\K_0})$ as $K\rightarrow\infty$. Following the proof of \cite[Theorem 4.1 p.354]{EK09}, for all but countably many $\K<\K_0$, $(\tau^{K}_{\K})_K$ converges in law to  $\tau_{\K}$. Indeed,  from \cite[Proposition~2.11, Chapter VI]{JS}, the hitting time $\tau^{K}_{\K}$ is a continuous function of the process $\Xi_{\K_0}^{K}$, except for   discontinuity points of $\Xi_{\K_0}^{K}$. This set of points is at most countable, see \cite[Lemma~2.10 b), Chapter VI]{JS}.\\
%\cmb{J'ai changé; ca te va? }
Consequently, for all but countably many $\K< \K_0$ and $r\in \{S,M\}$, we have for any continuous and bounded function $f$ on $[0,a_{\infty})$, 
\begin{align}
&\lim_{K \to \infty} \int_0^{\tau_\K^K} \int_{[0,a_{\infty})} f(a) \mathcal{Y}^K(s,\{r\},da)ds
%\nonumber \\ &\qquad \qquad \qquad 
= \int_0^{\tau_{\K}} \int_{[0,a_{\infty})} f(a) \Gamma_{\K_0}(ds, \{r\},da). \label{Eq1}
 \end{align} 
Using arguments of \cite[Lemma~2.9]{KKP}, which can be applied thanks to integrability condition of Assumption~\ref{hyp:alpha} and Lemma \ref{lem:boundtightness-renormalise}, this convergence already holds for continuous space-time function $f: [0,T] \times [0,a_{\infty}) \mapsto \mathbb{R}$ which are dominated by $(1+\alpha_r)$.
%we have
% \begin{align}
% &\lim_{K \to \infty} \int_0^{\tau_\K^K} \int_{[0,a_{\infty})} f(a) \tauxinter_r(\Xi^{K}_{\K_0}(s),a) \mathcal{Y}^K(s,\{r\},da)ds
%\nonumber \\
% &\qquad \qquad \qquad = \int_0^{\tau_{\K}} \int_{[0,a_{\infty})} f(a)\tauxinter_r(\Xi_{\K_0}(s),a) \Gamma_{\K_0}(ds, \{r\},da).  \label{Eq2}
%\end{align}
%\cmv{J'ai essaye de garder le meme espace qu'avant, pour ne pas trop changer, donc je n'ai pas pris a support compact, ok ? si oui, il me semble que ce serait bon de garder ce point de vue et de changer derriere pour $C^{1,b}$ non ?}
Let us use Lemma~\ref{lem:mart-avecscaling} with $g=0$ and $f\in \mathcal C^{1,b}(\mathcal X)$ such that $f(M,\cdot)=0$. Writing $f(S,\cdot)=f \in  \mathcal C^{1,b}([0,a_{\infty})) $,
%  $f(S,\cdot) =f$ be some test function $f\in C^1_c(\mathbb{R}_+)$ to find that $M^K$, defined for all $t\geq 0$ by
\begin{align*}
M^K(t) 
&=\frac{1}{\lambda_K} \left\{\int_{[0,a_{\infty})} f (a) \mathcal Y^{K}_{\K_0}(t\wedge  \tau_{\K}^K,\{S\},da) - \int_{[0,a_{\infty})}  f (a) \mathcal Y^{K}_{\K_0} (0,\{S\},da)\right\} \\
&\qquad \quad -  \int_0^{t\wedge \tau_{\K}^K} \int_{\mathcal X}  H(\Xi^{K}_{\K_0}(s),r,a) \Gamma^{K}_{\K_0}(ds,dr,da) \\
% -  \int_0^{t\wedge \tau_r^K}\int_0^\infty  H_{S,M}(\Xi^{r_0,K}(s),a) \Gamma_M^{r_0,K}(ds,da)\\
&\qquad \quad +\frac{1}{\lambda_K} \int_0^{t\wedge \tau_{\K}^K} \int_{[0,a_{\infty})} (\tauxnaissance_S(a)-\tauxmort_S(a)) f(a) \mathcal Y^{K}_{\K_0}(s,\{S\},da)ds,
\end{align*}
is a square integrable martingale, where
\begin{equation}
\label{eq:H}
H(x, S ,a) =  \partial_a f(a) -\tauxinter_S(x,a) f(a), \quad 
H(x, M,a)=  \tauxinter_M(x,a) f(0).
\end{equation}
Using \eqref{Eq1} guarantees that we have the following convergence in law
\begin{align*}
&\lim_{K \to \infty} \int_0^{t\wedge \tau_{\K}^K} \int_{\mathcal X}  H(\Xi^{K}_{\K_0}(s),r,a) \Gamma^{K}_{\K_0}(ds,dr,da)\\
&\qquad \qquad \qquad = \int_0^{t\wedge \tau_{\K}} \int_{\mathcal X}  H(\Xi_{\K_0}(s),r,a) \Gamma_{\K_0}(ds,dr,da).
\end{align*}
Besides \eqref{dominnb} ensures that 
\begin{align}
& \lambda_K^{-1}\left| \int_{[0,a_{\infty})}f (a) \mathcal Y^{K}_{\K_0}(t,\{S\}, da) - \int_{[0,a_{\infty})} f (a)  \mathcal Y^{K}_{\K_0}(0,\{S\}, da) \right| \label{eq:uneetapeparmitantdautre}  \\
 &\qquad \qquad \qquad \qquad  \qquad \qquad \qquad \qquad  \qquad \qquad \leq \frac{C \|f\|_\infty}{\lambda_K} \,  \sup_{t\leq T}   \mathcal Y^{K}_{\K_0}(t,\{S\}, [0,a_{\infty})),\nonumber
\end{align}
which tends to $0$, in probability, as $K\to \infty$. Similarly, in probability,
$$
\lim_{K \to \infty} \frac{1}{\lambda_K} \int_0^{t\wedge \tau_{\K}^K} \int_{[0,a_{\infty})} (\tauxnaissance_S(a)-\tauxmort_S(a)) f(a) \mathcal Y^{K}_{\K_0}(s,\{S\},da)ds = 0.
$$

Combining this three last convergence results, we obtain that $M^K$  converges in law to $M$, given, for all $t\geq 0$, by 
$$
M(t)=-  \int_0^{t\wedge \tau_{\K}} \int_{\mathcal X}  H(\Xi_{\K_0}(s),r,a) \Gamma_{\K_0}(ds,dr,da).
$$
Process $M$ remains a martingale. It is also a.s. Lipschitz because function $H$ is bounded. Consequently, it is null. We have then proved that for every $t\geq 0$ and 
$f\in \mathcal C^{1,b}([0,a_{\infty}))$, we have
\begin{equation}
\label{eq:annuleHS}
 \int_0^{t\wedge \tau_{\K}} \int_0^\infty H(\Xi_{\K_0}(s),r,a) \Gamma_{\K_0}(ds,dr,da) =0\, \text{ a.s.}
\end{equation}
for $H$ defined in \eqref{eq:H}. Now, thanks to \cite[Lemma 1.4]{K92}, we can decompose $\Gamma_{\K_0} (ds, \{S\},da)$ as $\Gamma_{\K_0} (ds, \{S\},da) = \gamma_{\K_0}(s,\{S\},da)\Lambda_S(ds)$. As \eqref{eq:annuleHS} holds for every $t\geq 0$,  focusing on functions $f$ such that $f(0)=0$,
we obtain a.s. and for $\Lambda_S$-almost all $s\leq t\wedge \tau_\K$,
\begin{equation*}
%\label{eq:annuleHS2}
 \int_0^\infty H(\Xi_{\K_0}(s),S,a) \gamma_{\K_0}(s,\{S\}, da) =0.
\end{equation*}
 In conclusion, for every $f\in\mathcal C^{1,b}([0,a_{\infty}))$ such that $f(0)=0$ and for $\Lambda_S$-almost all $s\leq t\wedge \tau_{\K}$, we almost surely have
\begin{equation}
\label{eq:annuleHS2}
 \int_{[0,a_{\infty})} (\partial_a f_S(a) -\tauxinter_S(\Xi_{\K_0}(s),a) f(a)) \gamma_{\K_0}(s,\{S\},da) =0.
\end{equation}
Let us show now that this functional equation imposes the form of $\gamma_{\K_0}$ through the solutions of the associated Poisson Equation. We proceed with a fix realization of the process and the results hold a.s. Consider $s\leq t\wedge \tau_{\K}$. For 
 any test function $g\in C^1_c([0,a_{\infty}))$ such that
$$
\int_0^\infty g(v) p_S(\Xi_{\K_0}(s),v) dv=0,
$$
the  function $f$ defined by
$$
f : a\mapsto p_S(\Xi_{\K_0}(s),a)^{-1} \int_0^a g(v) p_S(\Xi_{\K_0}(s),v) dv%\\
%&= - e^{\int_0^a \tauxinter_S(x,u) du} \int_a^\infty g_S(v)  e^{-\int_0^v \tauxinter_S(X(s),u) du} dv
$$
 is well-defined  for each fixed $s$ and belongs to $\mathcal C^{1,b}(\mathcal X)$. This function verifies $f(0)=0$ and is solution of the Poisson equation:
$$
\forall a\in [0, a_\infty), \qquad  \partial_a f(a) -\tauxinter_S(\Xi_{\K_0}(s),a) f(a)= g(a) \text{ a.s. }  %\qquad  H_{S,M}(s,x,a)=0
$$
By \eqref{eq:annuleHS2}, it yields 
$$
\int_{[0,a_{\infty})} g(a)\, \gamma_{\K_0}(s,\{S\},da) =  0.
$$ 
We extend this identity to $g\in C^1([0,a_{\infty}))$ such that 
$\int_0^\infty g(v) p_S(\Xi_{\K_0}(s),v) dv=0$ by an approximation argument.
We   can then apply this identity  to  $g:a \mapsto h(a) - \int_{[0,a_{\infty})} h(v) p_S(\Xi_{\K_0}(s),v)dv$ for any $h\in C^1([0,a_{\infty}))$. We   obtain that
 $p_S(\Xi_{\K_0}(s), \cdot)$ is  the density of the measure $\gamma_{\K_0}(s,\{S\},\cdot)$ with respect to Lebesgue measure. Hence,
\begin{equation}
\label{eq:gammaStmp}
\Gamma_{\K_0}(ds,\{S\},da) =   \gamma_{\K_0}(s,\{S\},[0,a_\infty)) \frac{p_S(\Xi_{\K_0}(s),a)}{\int_0^\infty p_S(\Xi_{\K_0}(s),w) dw}  \Lambda_S(ds)da.
\end{equation}
Similarly, we can prove
\begin{equation}
\label{eq:gammaMtmp}
\Gamma_{\K_0}(ds,\{M\},da) =   \gamma_{\K_0}(s, \{M\} , [0,a_\infty))  \frac{p_M(\Xi_{\K_0}(s),a)}{\int_{[0,a_{\infty})} p_M(\Xi_{\K_0}(s),w) dw}\Lambda_M(ds)da.
\end{equation}
Now, using  \eqref{eq:annuleHS} with  $f\equiv 1$ yields for every $t\geq 0$,
$$
\int_0^{t\wedge \tau_{\K}} \int_{[0,a_{\infty})}\tauxinter_S(a,\Xi_{\K_0}(s))  \Gamma_{\K_0}(ds,\{S\},da) = \int_0^{t\wedge \tau_\K} \int_{[0,a_{\infty})} \tauxinter_M(a,\Xi_{\K_0}(s))  \Gamma_{\K_0}(ds,\{M\},da).$$
This implies the following equality of measures
$$
\int_{[0,a_{\infty})}\tauxinter_S(a,\Xi_{\K_0}(s))  \Gamma_{\K_0}(ds,\{S\},da) =\int_{[0,a_{\infty})}\tauxinter_M(a,\Xi_{\K_0}(s))  \Gamma_{\K_0}(ds,\{M\},da).
$$
Integrating ~\eqref{eq:gammaStmp} and ~\eqref{eq:gammaMtmp} over $[0,a_\infty)$ and using the previous equality, we obtain
$$
  \frac{ \gamma(s,\{S\},[0,a_\infty))}{\int_0^\infty p_S(\Xi_{\K_0}(s),w) dw} \Lambda_S(ds) =   \frac{ \gamma(s,\{M\},[0,a_\infty))}{\int_0^\infty p_M(\Xi_{\K_0}(s),w) dw} \Lambda_M(ds).
$$
Finally, we have
$$
  \gamma(s,\{r\},[0,a_\infty)) \Lambda_r(ds) =  \frac{\int_{[0,a_{\infty})} p_r(s,w) dw}{\overline{p}(\Xi_{\K_0}(s))}\, \Gamma_{\K_0}(ds,\{S,M\}, [0,a_\infty)),
$$
and
$$
  \Gamma(ds,\{r\},da)= \frac{p_r(\Xi_{\K_0}(s),a)}{\overline{p}(\Xi_{\K_0}(s))} \,\Gamma_{\K_0}(ds,\{S,M\}, [0,a_\infty))\,  da.
$$
 %$$\tauxnaissance_S(s,\R_+) \Lambda_S(ds) = y \frac{\int_0^\infty p_S(s,w) dw}{\int_0^\infty p_M(s,w) dw+ \int_0^\infty p_S(s,w)  dw} ds
%$$
It ends the proof.
\end{proof}

Let us now focus on the number of preys and the whole number of predators. We prove that
limiting points satisfy the ODE \eqref{ODElimit}.
 \begin{lem}
 \label{lem:x}
Let $T>0$ and $\K_0>0$ and $(\Xi_{\K_0},\Gamma_{{\K_0}})$   be a limiting point of
$(\Xi_{\K_0}^K,\Gamma_{\K_0}^K)_K$ in $\mathbb{D}([0,T], \mathbb{R}_+) \times \mathfrak{M}([0,T]\times \mathcal X)$. For all but countably many $\K<\K_0$, the measure $\mathbf{1}_{s\leq \tau_{\K}}\Gamma_{\K_0}(ds,\{S,M\}, [0,a_\infty))$ admits a density $\mathcal{Y}_{\K_0}$ with respect to the Lebesgue measure, which verifies, for all $t\geq 0$, 
\begin{align*}
\Xi_{\K_0}(t\wedge \tau_{\K})&= \Xi_{\K_0}(0) + \int_0^{t\wedge \tau_\K} \left((\gamma-\beta)\, \Xi_{\K_0}(s) -\mathcal{Y}_{\K_0}(s) \phi(\Xi_{\K_0}(s)) \right) ds\\
\mathcal{Y}_{\K_0}(t\wedge \tau_\K) &= \mathcal{Y}_{\K_0}(0 ) + \int_0^{t\wedge \tau_\K} \mathcal{Y}_{\K_0}(s) \psi(\Xi_{\K_0}(s)) ds.
\end{align*}
\end{lem}

\begin{proof}
As in Lemma~\ref{lem:gamma}, to avoid the use of sub-sequences, we assume that
$(\Xi_{\K_0}^K,\Gamma_{\K_0}^K)$ converges to $(\Xi_{\K_0},\Gamma_{{\K_0}})$ in law. We use again Lemma~\ref{lem:mart-avecscaling}, with now $f\equiv 0$ and $g\equiv \text{Id}$. It ensures that  $M^K$, defined for every $t\geq 0$ by
\begin{align*}
M^K(t)
=\Xi_{K_0}^K(t\wedge \tau_{\K}^K) &-\Xi_{\K_0}^K(0)
+ \int_0^{t\wedge \tau_{\K}^K}  (\gamma-\beta)\, \Xi_{\K_0}^K(s) ds\\
&-\int_0^{t\wedge \tau_{\K}^K}\int_0^{\infty}  \tauxinter_S(a,\Xi_{\K_0}^{K}(s)) \Gamma_{\K_0}^K(ds, \{S\},da),
\end{align*}
is a square integrable martingale. It then converges in law to $M$, defined for every $t\geq 0$ by 
\Bea
M(t)%&=& \Xi^{r_0}(t\wedge \tau_{r})-\Xi(0)+\int_0^t \int_{[0,a_{\infty})} \tauxinter_S(a,\Xi^{r_0}(s))\Gamma^{r_0}(ds, \{S\},da)\\
=\Xi_{\K_0}(t\wedge \tau_{\K})&-&\Xi_{\K_0}(0)+\int_0^{t\wedge \tau_\K} (\gamma-\beta)\, \Xi_{\K_0}(s)ds\\
&-&\int_0^{t\wedge \tau_\K} \phi(\Xi_{\K_0}(s))\Gamma_{\K_0}(ds,\{S,M\}, [0,a_\infty)).
\Eea
%\cmb{J'ai ajouter local, car le seul doute peut être la pour moi, mais apres on tronque. Voir par exemple \url{https://mathoverflow.net/questions/83304/when-is-the-limit-of-martingales-a-martingale}La reponse de George Lowther le 13'11 à 2:55; j'aurai fait pareil}
Besides $M$ is a local martingale.
Similarly and as computed in the proof of Lemma~\ref{lem:tightness}, the bracket of $\big(M_{t\wedge \tau_{\K}^K}^K\big)_{t\geq 0}$  converges to $0$ in probability and then  $M_{t\wedge{\tau_\K}}=0$ a.s. It proves the first part of the result. \\

 We need now to describe  $\Gamma^{\K_0}(ds,\{S,M\}, [0,a_\infty))$. Again, we apply Lemma~\ref{lem:mart-avecscaling} but now with $f\equiv 1$ and $g\equiv 0$, to obtain that $M^K$ defined for all $t\geq 0$ by 
\begin{align*}
M^K(t)&=\mathcal Y^K(t\wedge \tau^K_\K,[0,a_\infty)) -\mathcal Y^K(0,[0,a_\infty))\\
&\quad +\int_0^{t\wedge \tau_\K^K}  \int_{[0,a_{\infty})} \tauxmort_S(a)\Gamma_{\K_0}^{K}(ds, \{S\},da)+\int_0^{t\wedge \tau_{\K}^K} \hspace{-0.3cm} \int_{[0,a_{\infty})} \tauxmort_M(a)\Gamma_{\K_0}^{K}(ds,\{M\},da)
\end{align*}
is a square integrable martingale and  
%square integrable martingale and its bracket
%\cmv{Il y a un melange $\K$, $\K_0$, au moins dans la definition de $M^K$. Je pense que c'est correct car ca revient a arreter au final au premier temps d'arret, mais a verifier. Les points rouhes ont ete modifies}\cmb{pas compris}
\begin{align*}
\langle M^K\rangle(t)&=
%&\quad \, +\lambda_K\int_0^{t\wedge \tau_{\K}^K}  \sum_{i \in \mathcal P(s)} \tauxinter(r_i(s),a_i(s),\Xi^K(s))\\
%&\quad \qquad  \times \left(\frac{1}{K_2}(f(\overline{r_i}(s),0)-f(r_i(s),a_i(s))\right)^2ds\\
%&&+\int_0^{t}  \sum_{i \in \mathcal P_M(s)} \tauxinter_M(a_i(s-),X(s-))\left(f_S(0)-f_M(a_i(s))\right)^2ds\\
\frac{1}{K_2^2} \int_0^{t\wedge \tau_{\K}^K}  \sum_{i \in \mathcal P(s)} \left(\tauxnaissance(r_i(s),a_i(s))+\tauxmort(r_i(s),a_i(s)) \right) ds.
\end{align*}
Using \eqref{dominnb}, it ensures that $\E(\langle M^K\rangle(t\wedge \tau_{\K}^K))$   
converges to $0$. Consequently, the process  $\big(\int_0^{t \wedge {\tau_\K}} M^K(t) dt\big)_{t\in [0,T]}$ tends in law to $0$ in $\mathbb D([0,T], \R_+)$ as $K\to \infty$. It also tends to
\begin{align*}
0 &= \Gamma_{\K_0}([0,T\wedge \tau_\K],\{S,M\}, [0,a_\infty))%\overline{\Gamma}^{r_0}([0,T\wedge \tau_r])%\Gamma^{r_0}(ds,\{S,M\}, [0,a_\infty))
-\mathcal Y(0,[0,a_\infty))(T\wedge \tau_\K)\\
&\qquad \qquad \qquad \qquad  \qquad \qquad + \int_0^{T \wedge \tau_\K}  \int_\mathcal{X} \tauxmort_S(a) \Gamma_{\K_0}(ds,dr,da) dt.
\end{align*}
%it is now run-of-the-mill to obtain that it vanishes from finite variation arguments.
Using Lemma~\ref{lem:gamma} and the definition of $\psi$,  it yields
\begin{align*}
&\Gamma_{\K_0}([0,T\wedge \tau_\K],\{S,M\}, [0,a_\infty))\\
&\qquad \qquad =\mathcal Y(0,\{S,M\},[0,a_\infty))-\int_0^{T\wedge \tau_\K} \int_0^t \psi\left(\Xi_{\K_0}(s)\right) \Gamma_{\K_0}(ds,\{S,M\}, [0,a_\infty) ) dt.
\end{align*}
This means that the measure $\mathbf{1}_{s\leq \tau_\K}\Gamma_{\K_0}(ds,\{S,M\}, [0,a_\infty))$ has a density $\mathcal{Y}$ with respect to the Lebesgue measure defined for all $t\geq 0$ by
\begin{align*}
\mathcal{Y}(t)
%&= \mathcal Y(0,\{S,M\},[0,a_\infty)) - \int_0^t \psi\left(\Xi_{\K_0}(s)\right) \overline{\Gamma}_{\K_0}(ds)\\
 & = \mathcal Y(0,\{S,M\},[0,a_\infty)) - \int_0^t \psi\left(\Xi_{\K_0}(s)\right) \mathcal{Y}(s) ds.
\end{align*}
It is the desired result.
\end{proof}

\begin{proof}[Proof of Theorem~\ref{thm:main}] Let  $(x_0,y_0)\in (\R_+^*)^2$ be the initial condition of $(x,y)$.
 Assumption~\ref{eq:systeme-edo} guarantees that
 %, the ODE system admits a unique global solution $(x,y)$ on $(\R_+^*)^2$ with initial condition $(x_0,y_0)$. Consequently,
  for any time horizon time $T>0$, there exists $\K_0>0$ such that for all $t\leq T$, $x(t)\in (1/\K_0,\K_0)$. Let $(\Xi_{\K_0},\mathcal Y_{\K_0},\Gamma_{{\K_0}})$ be any limiting values of $(\Xi_{\K_0}^{K}, \mathcal Y^K_{\K_0} ,\Gamma_{{\K_0}}^{K})$  in $\mathbb{D}([0,T], \mathbb{R}_+)^2  \times \mathfrak{M}([0,T]\times \mathcal X)$. By continuity of $x$, we can choose some $\K<\K_0$ such that conclusion of Lemma~\ref{lem:gamma} and Lemma~\ref{lem:x} hold and $x(t) \in (1/\K,\K)$ for any $t\leq T$. Consequently, $(\Xi_{\K_0}, \mathcal{Y}_{\K_0})$ and  $(x,y)$ satisfy the same evolution equation \eqref{ODElimit} on time interval $[0, T\wedge \tau_\K]$. Uniqueness guaranteed 
by Assumption~\ref{eq:systeme-edo} ensure that they coincide up to time $T\wedge \tau_\K$. It follows that 
$\tau_\K\geq T$ because $x(t)$ belongs to $(1/\K, \K)$ for any $t\leq T$. 

By Lemma~\ref{lem:gamma} and Lemma~\ref{lem:x}, we also have that 
$$
\Gamma_{{\K_0}}(dt,\{r\}, da)=y(t)p_r(x(t),a) \phi(x(t)) \, dt\, da.
$$
Besides,  for any continuous and bounded function $g$, we have both
$$
\int_0^{T\wedge \tau_\K} g(t)  \mathcal Y^K_{\K_0}(t) dt \stackrel{K\rightarrow \infty}{\longrightarrow} 
\int_0^{T\wedge \tau_\K} g(t)  \mathcal Y_{\K_0}(t) dt
$$
and 
$$
\int_0^{T\wedge \tau_\K} g(t)  \mathcal Y^K_{\K_0}(t) dt  \stackrel{K\rightarrow \infty}{\longrightarrow}   \int_0^{T\wedge \tau_\K} g(t) y(t) dt.
$$
since $  \sum_{r\in\{S,M\}} \int_0^{a_\infty} p_r(x(t),a) \phi(x(t)) \, da  =1$.
It ensures that
%\cmb{Je suppose que tt ce qui est en rouge veut dire que tu l'as modifié ? Car tt est correct, je ne vois pas ce que je peux/dois corriger}
$$\mathcal Y_{\K_0}(t)= y(t) \qquad \text{for almost every } \, t\geq 0.$$ 
As trajectories are \textit{càdlàg}, this identity holds for every $t\geq 0$. Using now Lemma~\ref{lem:tightness}, it  ensures the convergence of $(\Xi^{K},\mathcal Y^K(\cdot ,\{S,M\},[0,a_\infty)), \Gamma^{K})$ over $[0,T]$ to $(x, y(t)p_r(x(t),a) \phi(x(t)) \, dt\, da)$ in $\mathbb{D}([0,T], \mathbb{R}_+)^2  \times \mathfrak{M}([0,T]\times \mathcal X)$.
\end{proof}

\section{Examples and comments}
\label{se:exemple}
In this section, we illustrate and apply our convergence results
to examples motivated from ecology. We both recover classical limiting dynamical systems and functional responses and consider some new cases.
We do not discuss of technical points here. Time distributions  of interactions considered in this section satisfy Assumption~\ref{hyp:alpha}.  It can be checked
by using $\lyap: a \mapsto a^{\epsilon}$ or $\lyap: a \mapsto (a_\infty-a)^{-(1+\epsilon)}$ for some $\epsilon>0$. 

Let us first recall that  $T_S(x)$ and $T_M(x)$ are the random time for searching and manipulating
when the density of preys is $x$ preys.
As expected and seen above, the macroscopic death rate of preys  induced by predation is
$$\phi(x)=  \frac{1}{\E[T_S(x)]+\E[T_M(x)]}.$$
Besides, writing
$$\tauxgrowth_S(a)= \tauxnaissance_S(a) - \tauxmort_S(a), \quad \tauxgrowth_M= \tauxnaissance_M- \tauxmort_M$$
 the growth rate of the population size of predators is
\begin{align*}
\psi(x)%&=\phi(x) \E\left[ \int_0^{T_S(x)} (\tauxnaissance_S(u) - \tauxmort_S(u))) du \, + \, \int_0^{T_M(x)} (\tauxnaissance_M(u) - \tauxmort_M(u)) du\right] \\
&=\phi(x) \E\left[ \int_0^{T_S(x)} \tauxgrowth_S(a) da \, + \, \int_0^{T_M(x)} \tauxgrowth_M(a) da\right].
\end{align*}
%where $\tauxnaissance_r(u)$ (resp. $\tauxmort_r(u)$) is the birth (resp. death) rate of
%predators which are in status $r$ from time $u$ and $\tauxgrowth_r= \tauxnaissance_r - \tauxmort_r$.
%We will present some forms of these functions from the more classical setting to the more new ones.

\subsection{Classical setting and functional responses : memory less interactions}
 \label{se:exmarkov}

Let's start by the classical case where memory less property is assumed for each component of the dynamic (interaction, birth, death). Times involved are then exponential.
This assumption is probably not realistic for manipulating time in general.
For searching time
it can be justified with the hypothesis of rapid mixing of the preys in the medium where predators live. %If we further assume independence between preys, we
 %Then, no additional age structure is needed and the process counting the number of preys and predators. 
In this case, the growth rate $\psi$ of predators simplifies as
\begin{align*}
\psi(x)
&%= \frac{(\tauxnaissance_S-\tauxmort_S) \E[T_S(x)]+ (\tauxnaissance_M-\tauxmort_M) \E[T_M(x)]}{\E[T_S(x)]+\E[T_M(x)]}
= \frac{\tauxgrowth_S \E[T_S(x)]+ \tauxgrowth_M \E[T_M(x)]}{\E[T_S(x)]+\E[T_M(x)]}.
\end{align*}
%\cmv{J'ai supprim\'e : "In particular, $\phi$ and $\psi$ only depend on the law of $T_S, T_M$ through their mean" qui me semblait tenir de Lapalisse, vu que ces va ont un parametre...}
%\cmb{Cette phrase signifiait que si on choisit deux lois différentes pour $T_M$ qui possède la même moyenne alors on retrouve le même $\phi$, ce qui est évident via la forme mais remarquable, non?}
 We recover some classical functional responses with usual supplementary assumptions : 
\begin{itemize}
\item[$i)$] \emph{ No manipulation   and  search time inversely proportional
to the density:}
 $$T_M(x)=0 \qquad \E[T_S(x)]=\frac{1}{cx},$$
 for some $c>0$. 
 This assumption is justified for instance where rapid mixing allows to say that each prey meets independently  the predator with a small probability after an exponential time, since the minimum of independent exponential variables is exponential distributed and its parameter is the sum of each parameter.\\
 It leads to the classical Holling type I functional response  and Lotka-Volterra form for the consumption of preys (while its counterpart for predators is linear instead of bilinear):
\begin{equation}
\label{eq:exemple-notreLVpresque}
\phi(x)=c\, x, \qquad \psi(x)=(\tauxnaissance_S-\tauxmort_S).
\end{equation}
Let us mention that in our framework we have assumed that $\E(T_M(x))$
is lower bounded on compact set, which excludes the degenerated case $T_M=0$. But 
our proofs extend directly to this situation  (with a single status instead of two)
or can be obtained, at the limit, by letting $\E(T_M^{\varepsilon}(x))$ decrease to $0$ as $\varepsilon \rightarrow 0$.
%\textcolor{red}{Note that in our framework, we cannot really choose $T_M=0$, but it would not really be a problem.}
%\cmv{Dans l'absolu, on a pas le droit \`a $T_M=0$ vu nos hypotheses de construction, mais tout converge...a voir}

\item[$ii)$] \emph{Fixed mean manipulation time  and  search time inversely proportional
to the density:}
$$\E[T_M(x)]=t_0>0, \qquad  \E[T_S(x)]=\frac{1}{cx}$$
for some $c>0$. It leads to the classical Holling type II functional response  and Rosenzweig MacArthur / Monod model :
$$
\phi(x)=\frac{cx}{1+t_0cx}, \qquad \psi(x) = \tauxgrowth_S +  (\tauxgrowth_M - \tauxgrowth_S) t_0  \frac{cx}{1+t_0cx}.
$$
Constant $(\tauxgrowth_M - \tauxgrowth_S) t_0$ is related to the  "yield constant" in microbial ecology, as in the chemostat equation for instance. 
%\cmv{Et lui, on ne lui donne pas son $\phi$ et $\psi$ ?}
%\cmb{Comme on ne justifie pas ces choix de moments à partir de la loi expo, je l'ai changé de section, ca rend aussi les choses + homogènes}
\item[$iii)$] \emph{Fixed mean manipulation time  and generalist predator.}
Another usual response  make the searching time of the prey increase faster with low density since the predator may dedicate more time to other species. This can leads to assuming the following assumption:
 $$\E[T_M(x)]=t_0>0, \qquad \E[T_S(x)]=\frac{1}{cx^2}.$$
 This ensures the Holling type III functional response. To describe this generalist behavior of the predator more precisely, this  would require to consider additional species in our model, see also \cite{BBC18} for instance.
\end{itemize}

\subsection{A first generalization : non-exponential time of interaction}

 % This is not directly captured in our setting (to avoid superfluous notation) but should work without difficulty.
In the models considered in the previous section, without fast mixing, we do not expect that searching time is  exponentially distributed. 
We refer e.g. \cite{BBC18} for some simple models where the law of the searching
time is described or 
to \cite{Duijns} for some data. \\
For instance, if preys are spatially uniformly distributed with fixed positions,
a toy model \cite[Page 11]{BBC18} with motion towards the nearest prey leads to
$$
\mathbb{E}[T_S(x)] = \frac{c}{\sqrt{x}}.
$$
Besides, the manipulating is not expected to be exponentially distributed either.

A first consequence of our results is that we can extend the convergence
results to this non-exponentially distributed times. From the point of view of prey consumption
and at the first order of the macroscopic scale, the form of the distribution has no effect (beyond its mean). Let us turn to new effects due to non-exponential laws.

\subsection{Influence of distribution of time interaction}
The  consumption of prey at a first order macroscopic approximation is only sensitive 
to mean time of interactions trough the function $\phi$.
The impact of predation on the evolution of predators may be more subtle.

Let's give an explicit example. Assume that the individual growth rate is linked 
to the consumption of preys via the following age dependance 
$$
\forall a \leq 0, \ \tauxgrowth_S(a)= - A +  B e^{-Ca}.
$$
for some $A, B, C >0$. 
Assuming $A+B>0$, it models the fact that the  more a predator is waiting for a prey, the less it  (successively) reproduces and/or the fastest it dies. 
For sake of simplicity and following Section \ref{se:exmarkov}, let us consider the case when $T_S(x)$ has exponential distribution with mean $1/cx$, we have
\begin{align*}
\E\left[ \int_0^{T_S(x)} \tauxgrowth_S(u) du\right] 
%&= \E\left[ -A T_S(x) + \frac{B}{C} (1 -e^{-CT_S(x)})\right]\\
&= -\frac{A}{cx} + \frac{Bcx}{Ccx+1}.
\end{align*}
and, again in the setting of Subsection~\ref{se:exmarkov}, this gives
\begin{itemize}
\item[$i)$] Without manipulation, i.e. $T_M(x)=0$, we get
\begin{equation}
\label{eq:exemple-notreLV}
\phi(x)= cx, \qquad  \psi(x)=  -A + B \frac{(cx)^2}{Ccx +1}.
\end{equation}
In particular, $\psi(x)\rightarrow-A$ as $x\rightarrow 0$ and $\psi(x)\sim_{x\rightarrow \infty} \frac{Bc}{C} x$. That is $\psi(x)$ behaves as $-A+B'x$, with $A<0$  as in the Lotka-Volterra model. 
%This is even more the case when $B$ and $C$ are large but of the same order : $B/C \approx B'/c$. 
\item[$ii)$] With fixed positive manipulation, i.e. $t_0=\E(T_M (x))$
and $\lambda_M(x)=\lambda_M $:
$$
\phi(x)= \frac{cx}{1+t_0cx}, \qquad 
\psi(x) = \frac{cx}{1+t_0cx} \left(-\frac{A}{cx} + \frac{Bcx}{Ccx+1} + \lambda_M t_0 \right)
$$
Thus is $\psi(x) \rightarrow  A <0 $ as $x\rightarrow 0$ and $\psi(x) \rightarrow \frac{B}{C t_0} + c \lambda_M >0 $ as  $x\rightarrow \infty$. Then it behaves as classical Holling type II prey-predator model:
$$
\psi(x) = -A + \mu \frac{x}{x+K}.
$$
\end{itemize}
%\cmv{mais ce serait bien de montrer autre chose}
We then recover here the two classic forms without directly assuming a conversion of prey into predators.

\iffalse
Let's give an explicit example. Assume that for some $A, B>0$, we have 
$$
\forall u \leq 0, \ \tauxgrowth_S(u)= B - 2 A u,
$$
translating that more a predator search less it reproduces and more easily it dies. This rate can come from
$$
\tauxnaissance(u) = B, \qquad \tauxmort=2 A u.
$$

With the common modeling of interaction times as in Subsection~\ref{se:exmarkov}, that is $T_S(x)$ has exponential distribution with mean $1/cx$, we have
\begin{align*}
\E\left[ \int_0^{T_S(x)} \tauxgrowth_S(u) du\right] 
&= B \frac{1}{c x} - A \left(\frac{1  + c x}{c^2 x^2}\right).
\end{align*}
and, again in the setting of Subsection~\ref{se:exmarkov}, this gives
\begin{itemize}
\item Without manipulation:
$$
\phi(x)= cx, \qquad  \psi(x)=  \left(B - A\right)  -  \frac{A}{c x}.
$$
In particular, a too low prey rate implies a dramatic decrease for predators.
\item With manipulation (with constant $\lambda_M$):
$$
\phi(x)= \frac{cx}{1+t_0cx}
$$
and
$$
\psi(x)= \frac{ t_0 \lambda_M c^2 x^2 +  \left(B -A \right) cx - A  }{(1+t_0cx) cx}; 
$$
that is $\psi(x) \sim_{0} - A/ cx $ and $\psi(x) \sim_{\infty} \lambda_M $ similarly as the non-manipulation case.
\end{itemize}
\fi

\subsection{About the behavior of the limiting ODEs}

In this work, we show the relevance of dynamical system of the form :

\begin{equation*}
\left\{
\begin{array}{rcl}
x'(t)&=&(\tauxnaissance - \tauxmort)x(t)-y(t)\phi(x(t)),\\
y'(t)&=&y(t)\psi(x(t)).
\end{array}
\right.
\end{equation*}
Let us just give some hints on its long-time behavior, even if a large literature exists, for the study of Lotka-Volterra type systems, in which these remarks are detailed in more details.

On the first hand, when $\phi :x \mapsto cx$, for some $c>0$, then it can be easily shown that, whatever the function $\psi$, the function
$$
t\mapsto L(x(t),y(t))
$$
is constant with the law conservation 
$$
L(x,y) = \lambda \log(y) - cy  - \int_1^{\log(x)} \psi(e^u) du.
$$
In particular, when the curve $L(x,y) = L(x(0),y(0))$ is bounded then $(x,y)$ is periodic, this is for instance the case when \eqref{eq:exemple-notreLV} holds. In contrast, even if there is a conservation law, the choice \eqref{eq:exemple-notreLVpresque} leads the predator number going to infinity and prey number to $0$.

On the second hand, in more generality, an equilibrium point $(x^\star,y^\star)$ of this system verifies
$$
\psi(x^\star)=0, \qquad y^\star = \frac{\lambda x^\star}{\phi(x^\star)}.
$$
Natural assumptions are decreasing rates $a\mapsto \lambda_S(a), a\mapsto\lambda_M(a)$ and mean times of interaction $x\mapsto\mathbb{E}[T_S(x)]$ and $x\mapsto\mathbb{E}[T_M(x)]$. Unfortunately, under these assumptions, we cannot state uniqueness of an equilibrium points.

Moreover the Jacobian at this equilibrium is equal to
$$
J_{(x^\star,y^\star)}= \begin{pmatrix}
    (x^\star \phi(x^\star) -\phi'(x^\star))y^\star     & -\phi(x^\star) /x^\star \\ 
    \psi'(x^\star) y^\star       & 0 
\end{pmatrix}.
$$
When $(x^\star \phi(x^\star) -\phi'(x^\star))y^\star \neq 0$ (this naturally exclude $\phi :x \mapsto cx$) and $\psi(x^\star) \neq 0$ then the associated equilibrium is unstable. In others cases, it is not exponentially stable.
%  \cmb{Les valeurs propres de $(a,b;c,d)$ sont $1/2$ fois $a \pm \sqrt{a^2 + 4 bc}$}

\subsection{Discussion, scaling and extensions}
\label{sect:discussion}
In this work, we are interested in  cases where the  number of preys is much larger than the number of predators and  the time for interactions  is much shorter that the time to give birth or the time to die for preys and predator.
Besides the time for searching may impact the survival of offsprings (via natality rate)
or the death probability.\\
This seems reasonable for many interactions. For instance, fox-rabbit, wolf-deer/caribou, white bear-seal, bear-fish, bird-worm, where the time for searching is of order of days or a week, while reproduction is of order of a year for both (and several years for death).

%And at smaller scales (examples ?) \textcolor{red}{Tu veux que je te parle de microbes qui mangent du sucre? Les proies ne se reproduisent pas dans ce cas. Je peux trouver la consommation mais en faite elle ne s'arrete pas de manger et se divise toute les 20 minutes, donc ca reste à des echelles tres differentes}

For most of the species mentioned above, extension of the model to several preys for one predator and interference between several predators should be considered. Also adding the biological age or non-exponential clock for birth and death (season effect, maturity, menopause ...) are interesting points to address. We focused here on relaxing the memory less of properties of interactions.  Such extensions seems to be accessible via the framework and techniques developed here even if technicalities may fast increase.\\
 
 Determining stochastic fluctuations around the limiting deterministic system is a challenging and interesting problem. It is relevant in particular when population sizes are not very large.%\cmb{Tu donnes pas trop de détails sur ce que l'on va faire?}
 The variance of interaction times should appear to describe fluctuations and may be much larger than in the exponential case.
The averaging approach of \cite{KKP} provides a natural path for this issue via Poisson equation. Adapting the techniques to our infinite dimensional setting seems challenging.\\
%An approach by duality may give an alternative point of view and should be the object of a future work.\\

{\bf Acknowledgement}
This work  was partially funded by the Chair "Mod\'elisation Math\'ematique et Biodiversit\'e" of VEOLIA-Ecole Polytechnique-MNHN-F.X and ANR ABIM 16-CE40-0001.
The authors are also grateful to Sylvain Billiard and Marius Bansaye for stimulating discussions about prey-predator interactions.

\bibliographystyle{alpha}
\bibliography{ref}

\begin{thebibliography}{DKPvG15}

\bibitem[AKF11]{AKF}
Tal Avgar, Daniel Kuefler, and John~M Fryxell.
\newblock Linking rates of diffusion and consumption in relation to resources.
\newblock {\em The American Naturalist}, 178(2):182--190, 2011.

\bibitem[BBC18]{BBC18}
Sylvain Billiard, Vincent Bansaye, and J-R Chazottes.
\newblock Rejuvenating functional responses with renewal theory.
\newblock {\em Journal of The Royal Society Interface}, 15(146):20180239, 2018.

\bibitem[BDBS96]{BBRS}
Jos{\'e}~AM Borghans, Rob~J De~Boer, and Lee~A Segel.
\newblock Extending the quasi-steady state approximation by changing variables.
\newblock {\em Bulletin of mathematical biology}, 58(1):43--63, 1996.

\bibitem[Bil13]{B13}
Patrick Billingsley.
\newblock {\em Convergence of probability measures}.
\newblock John Wiley \& Sons, 2013.

\bibitem[BKPR06]{BKPR}
Karen Ball, Thomas~G. Kurtz, Lea Popovic, and Greg Rempala.
\newblock Asymptotic analysis of multiscale approximations to reaction
  networks.
\newblock {\em Ann. Appl. Probab.}, 16(4):1925--1961, 11 2006.

\bibitem[BM15]{BM15}
Vincent Bansaye and Sylvie M{\'e}l{\'e}ard.
\newblock {\em Stochastic models for structured populations}.
\newblock Springer, 2015.

\bibitem[BT10]{BT10}
Vincent Bansaye and Viet~Chi Tran.
\newblock Branching feller diffusion for cell division with parasite infection.
\newblock {\em arXiv preprint arXiv:1004.0873}, 2010.

\bibitem[CKBG14]{CKBG}
Ross Cressman, Vlastimil K{\v{r}}ivan, Joel~S Brown, and J{\'o}zsef Garay.
\newblock Game-theoretic methods for functional response and optimal foraging
  behavior.
\newblock {\em PLoS One}, 9(2):e88773, 2014.

\bibitem[Cos16]{C16}
Manon Costa.
\newblock A piecewise deterministic model for a prey-predator community.
\newblock {\em The Annals of Applied Probability}, 26(6):3491--3530, 2016.

\bibitem[DKPvG15]{Duijns}
Sjoerd Duijns, Ineke~E Knot, Theunis Piersma, and Jan~A van Gils.
\newblock Field measurements give biased estimates of functional response
  parameters, but help explain foraging distributions.
\newblock {\em Journal of Animal Ecology}, 84(2):565--575, 2015.

\bibitem[DS13]{DS}
JHP Dawes and MO~Souza.
\newblock A derivation of holling's type i, ii and iii functional responses in
  predator--prey systems.
\newblock {\em Journal of theoretical biology}, 327:11--22, 2013.

\bibitem[EK09]{EK09}
Stewart~N Ethier and Thomas~G Kurtz.
\newblock {\em Markov processes: characterization and convergence}, volume 282.
\newblock John Wiley \& Sons, 2009.

\bibitem[FM04]{FM04}
Nicolas Fournier and Sylvie M{\'e}l{\'e}ard.
\newblock A microscopic probabilistic description of a locally regulated
  population and macroscopic approximations.
\newblock {\em The Annals of Applied Probability}, 14(4):1880--1919, 2004.

\bibitem[HDB97]{HBR}
Gert Huisman and Rob~J De~Boer.
\newblock A formal derivation of the "beddington" functional response.
\newblock {\em Journal of theoretical biology}, 185(3):389--400, 1997.

\bibitem[IW14]{IW14}
Nobuyuki Ikeda and Shinzo Watanabe.
\newblock {\em Stochastic differential equations and diffusion processes}.
\newblock Elsevier, 2014.

\bibitem[JKT02]{JKT}
Jonathan~M Jeschke, Michael Kopp, and Ralph Tollrian.
\newblock Predator functional responses: discriminating between handling and
  digesting prey.
\newblock {\em Ecological Monographs}, 72(1):95--112, 2002.

\bibitem[JM86]{JM86}
Anatole Joffe and Michel M{\'e}tivier.
\newblock Weak convergence of sequences of semimartingales with applications to
  multitype branching processes.
\newblock {\em Advances in Applied Probability}, pages 20--65, 1986.

\bibitem[JS13]{JS}
Jean Jacod and Albert Shiryaev.
\newblock {\em Limit theorems for stochastic processes}, volume 288.
\newblock Springer Science \& Business Media, 2013.

\bibitem[KKP14]{KKP}
Hye-Won Kang, Thomas~G Kurtz, and Lea Popovic.
\newblock Central limit theorems and diffusion approximations for multiscale
  markov chain models.
\newblock {\em The Annals of Applied Probability}, 24(2):721--759, 2014.

\bibitem[Kur92]{K92}
Thomas~G Kurtz.
\newblock Averaging for martingale problems and stochastic approximation.
\newblock In {\em Applied Stochastic Analysis}, pages 186--209. Springer, 1992.

\bibitem[MT12]{MT12}
Sylvie M\'{e}l\'{e}ard and Viet~Chi Tran.
\newblock Slow and fast scales for superprocess limits of age-structured
  populations.
\newblock {\em Stochastic Processes and their Applications}, 122(1):250 -- 276,
  2012.

\bibitem[Tra06]{T06}
Viet~Chi Tran.
\newblock {\em {Mod{\`e}les particulaires stochastiques pour des probl{\`e}mes
  d'{\'e}volution adaptative et pour l'approximation de solutions
  statistiques}}.
\newblock Theses, {Universit{\'e} de Nanterre - Paris X}, December 2006.

\end{thebibliography}

\end{document}